\documentclass[a4paper,twopage,reqno,12pt]{amsart}
\usepackage[top=30mm,right=30mm,bottom=30mm,left=30mm]{geometry}

\usepackage{pdfsync}
\usepackage{tabularx}
\usepackage{graphicx}
\usepackage{amsmath}
\usepackage{amssymb}
\usepackage{amsfonts}
\usepackage{amsthm}
\usepackage{amstext}
\usepackage{amsbsy}
\usepackage{amsopn}
\usepackage{amscd}
\usepackage{enumerate}
\usepackage{color}
\usepackage[colorlinks]{hyperref}
\usepackage[hyperpageref]{backref}
\usepackage{multirow}
\usepackage{lipsum}
\usepackage{rotating}
\usepackage{longtable}
\usepackage{float}
\usepackage{lscape}
\usepackage{pdflscape}
\usepackage[]{algorithm2e}
\usepackage{url}

\newtheorem{theorem}{Theorem}[section]
\newtheorem{corollary}[theorem]{Corollary}

\newtheorem{lemma}[theorem]{Lemma}
\newtheorem{proposition}[theorem]{Proposition}

\theoremstyle{definition}

\newtheorem{remark}[theorem]{Remark}

\numberwithin{equation}{section}

\newcommand{\GU}{\mathrm{GU}}
\newcommand{\G}{\mathrm{G}}

\newcommand{\GL}{\mathrm{GL}}
\newcommand{\SL}{\mathrm{SL}}

\newcommand{\Sp}{\mathrm{Sp}}

\newcommand{\PSO}{\mathrm{PSO}}
\newcommand{\SO}{\mathrm{SO}}
\newcommand{\SU}{\mathrm{SU}}
\newcommand{\Sz}{\mathrm{Sz}}
\renewcommand{\O}{\mathrm{O}}

\newcommand{\PSL}{\mathrm{PSL}}
\newcommand{\PSU}{\mathrm{PSU}}

\newcommand{\PGU}{\mathrm{PGU}}
\newcommand{\PSp}{\mathrm{PSp}}

\newcommand{\PGL}{\mathrm{PGL}}
\newcommand{\PGaL}{\mathrm{P\Gamma L}}

\newcommand{\POm}{\mathrm{P \Omega}}
\newcommand{\J}{\mathrm{J}}

\newcommand{\A}{\mathrm{Alt}}
\newcommand{\W}{\mathrm{W}}
\newcommand{\E}{\mathrm{E}}
\newcommand{\Al}{\mathrm{A}}
\renewcommand{\S}{\mathrm{Sym}}
\newcommand{\Q}{\mathrm{Q}}

\newcommand{\C}{\mathrm{C}}
\newcommand{\D}{\mathrm{D}}
\newcommand{\V}{\mathrm{V}}
\newcommand{\Fi}{\mathrm{Fi}}

\newcommand{\Aut}{\mathrm{Aut}}

\newcommand{\Out}{\mathrm{Out}}

\newcommand{\PG}{\mathrm{PG}}

\newcommand{\B}{\mathrm{B}}
\newcommand{\N}{\mathrm{N}}
\newcommand{\M}{\mathrm{M}}
\newcommand{\F}{\mathrm{F}}

\newcommand{\Dmc}{\mathcal{D}}
\newcommand{\Bmc}{\mathcal{B}}

\newcommand{\Pmc}{\mathcal{P}}
\newcommand{\Cmc}{\mathcal{C}}
\newcommand{\Smc}{\mathcal{S}}

\newcommand{\e}{\epsilon}

\renewcommand{\leq}{\leqslant}
\renewcommand{\geq}{\geqslant}

\renewcommand{\mod}[1]{\ (\mathrm{mod}{\ #1})}
\newcommand{\imod}[1]{\allowbreak\mkern4mu({\operator@font mod}\,\,#1)}

\begin{document}
 \title[Flag-transitive Block designs with prime replication number]{Flag-transitive block designs with prime replication number and almost simple groups}

 \author[S.H. Alavi]{Seyed Hassan Alavi}%
 \thanks{Corresponding author: S.H. Alavi}
 \address{Seyed Hassan Alavi, Department of Mathematics, Faculty of Science, Bu-Ali Sina University, Hamedan, Iran.
 }%
 \email{alavi.s.hassan@basu.ac.ir and  alavi.s.hassan@gmail.com (G-mail is preferred)}
 \author[M. Bayat]{Mohsen Bayat}%
 \address{Mohsen Bayat, Department of Mathematics, Faculty of Science, Bu-Ali Sina University, Hamedan, Iran.}%
 \email{mohsen0sayeq24@gmail.com}
 \author[J. Choulaki]{Jalal Choulaki}%
 %\thanks{}
 \address{Jalal Choulaki, Department of Mathematics, Faculty of Science, Bu-Ali Sina University, Hamedan, Iran.}
  \email{j.choulaki@gmail.com}
 \author[A. Daneshkhah]{Asharf Daneshkhah}%
 %\thanks{}
 \address{Asharf Daneshkhah, Department of Mathematics, Faculty of Science, Bu-Ali Sina University, Hamedan, Iran.
 }%
  \email{adanesh@basu.ac.ir}
%  \author[Sh. Zang Zarin]{Sheyda Zang Zarin}%
 %\thanks{}
% \address{Sheyda Zang Zarin, Department of Mathematics, Faculty of Science, Bu-Ali Sina University, Hamedan, Iran.
% }%

 \date{\today}%
 %\dedicatory{}%
 %\commby{}%

\begin{abstract}
  In this article, we study $2$-designs with prime replication number admitting a flag-transitive automorphism group. The automorphism groups of these designs are point-primitive  of almost simple or affine type. We determine $2$-designs with prime replication number admitting an almost simple automorphism group.
\end{abstract}

\maketitle

\section{Introduction}\label{sec:intro}

A $2$-$(v,k,\lambda)$ design $\Dmc$ is a pair $(\Pmc,\Bmc)$ with a set $\Pmc$ of $v$ points and a set $\Bmc$ of blocks such that each block is a $k$-subset of $\Pmc$ and each two distinct points are contained in $\lambda$ blocks. We say $\Dmc$ is nontrivial if $2 < k < v-1$, and symmetric if $v = b$, where $b$ is the number of blocks of $\Dmc$.  Each point of $\Dmc$ is contained in exactly $r=bk/v$ blocks which is called the \emph{replication number} of $\Dmc$. An \emph{automorphism} of a symmetric design $\Dmc$ is a permutation of the points permuting the blocks and preserving the incidence relation. The full automorphism group $\Aut(\Dmc)$ of $\Dmc$ is the group consisting of all automorphisms of $\Dmc$. A \emph{flag} of $\Dmc$ is a point-block pair $(\alpha, B)$ such that $\alpha \in B$. For $G\leq \Aut(\Dmc)$, $G$ is called \emph{flag-transitive} if $G$ acts transitively on the set of flags. The group $G$ is said to be \emph{point-primitive} if $G$ acts primitively on $\Pmc$. A group $G$ is said to be \emph{almost simple} with socle $X$ if $X\unlhd G\leq \Aut(X)$, where $X$ is a nonabelian simple group. We adopt the standard notation as in \cite{b:BHR-Max-Low,b:Atlas} when we deal with  almost simple groups with socle finite classical simple groups, while in the case where the socle is an exceptional simple  group, we sometimes use alternative Lie notation. We here write $\A_{n}$ and $\S_{n}$ for the alternating group and the symmetric group on $n$ letters, respectively, and we denote by ``$n$'' the cyclic group of order $n$.
Further notation and definitions in both design theory and group theory are standard and can be found, for example, in \cite{b:Atlas,b:Dixon,b:KL-90,b:lander}.

The main aim of this paper is to study $2$-designs with flag-transitive automorphism groups.
In 1988, Zieschang \cite{a:ZIESCHANG} proved that if an automorphism group $G$ of a $2$-design with $\gcd(r,\lambda)=1$ is flag-transitive, then $G$ is point-primitive group of almost simple or affine type. Such designs admitting an almost simple automorphism group with socle being an alternating group, a sporadic simple group or a finite simple exceptional group have been studied in \cite{a:ABD-Exp-coprime,a:Zhou-lam-large-sporadic,a:Zhan-nonsym-sprodic,a:Zhou-nonsym-alt, a:Zhu-sym-alternating}. This problem for case where the socle is a finite simple groups of Lie type is still open. This paper is devoted to studying flag-transitive designs with prime replication number $r$ in which case $r$ and $\lambda$ are coprime. We know two infinite families of examples of designs with prime replication number namely projective space $\PG(n-1,q)$ with $(q^{n-1}-1)/(q-1)$ prime and Witt-Bose-Shrikhande space $\W(2^n)$ with $2^n+1$ a Fermat prime. The projective spaces are natural examples of designs with $2$-transitive automorphism groups \cite{a:Kantor-2-trans} and the latter example arises from studying flag-transitive linear spaces \cite{a:delan-linear-space}. Our main result is Theorem \ref{thm:main} below.

\begin{theorem}\label{thm:main}
Let $\Dmc$ be a nontrivial $2$-$(v, k, \lambda)$ design with prime replication number $r$, and let $\alpha$ be a point of $\Dmc$. If $G$ is a flag-transitive automorphism group of $\Dmc$ of almost simple type with socle $X$, then   %$X=\PSL_{n}(q)$, $G_{\alpha}\cap X=\,^{\hat{}}[q^{n-1}]{:}\SL_{n-1}(q){\cdot} (q-1)$, $v=(q^{n}-1)/(q-1)$ and $r$ is a primitive divisor of $(q^{n-1}-1)/(q-1)$ with $n\geq 3$, or
$\lambda\in\{1,2,3,5\}$ and $v$, $k$, $\lambda$, $X$, $G_{\alpha}\cap X$ and $G$ are as in one of the lines in {\rm Table~\ref{tbl:main}} or one of the following holds:
\begin{enumerate}[{\quad \rm (a)}]
\item $\Dmc$ is the Witt-Bose-Shrikhande space $\W(2^n)$ with parameters $v=2^{n-1}(2^{n}-1)$, $b=2^{2n}-1$, $r=2^{n}+1$ Fermat prime, $k=2^{n-1}$ and $\lambda=1$, for $n=2^{2^m}\geq 16$. Moreover, $G=X=\PSL_{2}(2^n)$ and  $G_\alpha \cap X=\D_{2(2^{n}+1)}$;
\item $X=\PSL_{n}(q)$, $G_{\alpha}\cap X=\,^{\hat{}}[q^{n-1}]{:}\SL_{n-1}(q){\cdot} (q-1)$, $v=(q^{n}-1)/(q-1)$ and $r$ is a primitive divisor of $(q^{n-1}-1)/(q-1)$ with $n\geq 3$.
%\item $X=\PSp_{4}(q)$, $G_{\alpha}\cap X=\PSp_{2}(q^2){:}2$, $v=q^2(q^{2}-1)/2$, $b=q^2(q^2+1)/2$, $r=q^2+1$ Fermat prime, $k=q^2-1$, and $\lambda=2$. Moreover, if $B$ is a block containing $\alpha$, then $G_B\cap X$  is isomorphic to either $^{\hat{}}\Sp_{2}(q)^{2}:2$, or $^{\hat{}}\SO_{4}^{+}(q)$.
\end{enumerate}
\end{theorem}
\begin{table}%[h]
\centering\scriptsize
\caption{Some nontrivial $2$-design with prime replication number}\label{tbl:main}
\resizebox{\textwidth}{!}{
\begin{tabular}{cllllllllll}
%{p{5mm}p{6mm}p{7mm}p{6mm}p{1mm}p{5mm}p{10mm}p{13mm}p{10mm}p{40mm}p{11mm}}
\hline
Line &
$v$ &
$b$ &
$r$ &
$k$ &
$\lambda$ &
$X$ &
$G_{\alpha}\cap X$&
$G$ &
Designs &
References \\ \hline
$1$ &
$6$ &
$10$ &
$5$ &
$3$ &
$2$ &
$\A_{5}$ &
$\D_{10}$ &
$\A_{5}$ &
- &
\cite{b:Handbook,a:Non-symmetric}\\
$2$ &
$7$ &
$7$ &
$3$ &
$3$ &
$1$ &
$\PSL_{2}(7)$ &
$\S_{4}$ &
$\PSL_{2}(7)$ &
$\PG(2, 2)$ &
\cite{a:ABD-PSL2,b:Handbook,a:reg-reduction} \\
$3$ &
$8$ &
$14$ &
$7$ &
$4$ &
$3$ &
$\PSL_{2}(7)$&
$7{:}3$ &
$\PSL_{2}(7)$ &
-&
\cite{b:Handbook,a:Non-symmetric}\\
$4$ &
$11$ &
$11$ &
$5$ &
$5$ &
$2$ &
$\PSL_{2}(11)$&
$\A_{5}$&
$\PSL_{2}(11)$ &
Paley &
\cite{a:ABD-PSL2,b:Handbook,a:reg-reduction} \\
$5$ &
$12$ &
$22$ &
$11$ &
$6$ &
$5$ &
$\M_{11}$ &
$\PSL_{2}(11)$ &
$\M_{11}$ &
- &
\cite{b:Handbook,a:Zhan-nonsym-sprodic}\\
$6$ &
$15$ &
$15$ &
$7$ &
$7$ &
$3$ &
$\A_{7}$ &
$\PSL_{2}(7)$ &
$\A_{7}$ &
$\PG_{2}(3,2)$&
\cite{a:ABD-PSL2,b:Handbook,a:Zhu-sym-alternating}\\
$7$ &
$15$ &
$35$ &
$7$ &
$3$ &
$1$ &
$\A_{7}$ &
$\PSL_{2}(7)$ &
$\A_{7}$ &
$\PG(3,2)$&
\cite{b:Handbook,a:Zhou-nonsym-alt}\\
$8$ &
$15$ &
$35$ &
$7$ &
$3$ &
$1$ &
$\A_{8}$ &
$2^3{:}\PSL_{3}(2)$ &
$\A_{8}$ &
$\PG(3,2)$&
\cite{b:Handbook,a:Zhou-nonsym-alt}\\
%
%$9$ &
%$\frac{q(q-1)}{2}$ &
%$q^{2}-1$ &
%$q+1$ &
%$\frac{q}{2}$ &
%$1$ &
%$\PSL_{2}(q)$ &
%$\D_{2(q+1)}$ &
%$\PSL_{2}(q)$ &
%Witt-Bose-Shrikhande space
%$\W(q)$, $q=2^{2^{m}}\geq 16$ &
%\cite{a:delan-linear-space, a:Non-symmetric} \\
%
\hline
\scriptsize
Note: & \multicolumn{10}{l}{\scriptsize The last column addresses to references in which a design with the parameters in}\\
 & \multicolumn{10}{l}{\scriptsize the line has been constructed.}
\end{tabular}
}
\end{table}

%\subsection{Outline of proof}\label{sec:outline}

In order to prove Theorem~\ref{thm:main} in Section~\ref{sec:proof}, we observe that the result for the case where the socle $X$ is an alternating group or a sporadic simple group follows immediately from \cite{a:Zhou-lam-large-sporadic,a:Zhan-nonsym-sprodic,a:Zhou-nonsym-alt,a:Zhu-sym-alternating}. For the remaining cases, the group $G$ is point-primitive \cite{a:ZIESCHANG}, or equivalently, the point stabiliser $H=G_\alpha$ is maximal in $G$. If $X$ is a finite classical simple group, then we apply Aschbacher's Theorem~\cite{a:Aschbacher} which says that $H$ lies in one of the eight geometric families $\Cmc_{i}$ ($i=1,\ldots8$) of subgroups of $G$, or in the family $\Smc$ of almost simple subgroups with some irreducibility conditions. The case where $X=\PSL_{2}(q)$ has been separately studied in Proposition \ref{prop:psl2}. For the remaining classical groups,
we use the minimal degrees of permutation representations of finite classical simple groups in characteristic $p$ to show that $r$ divides the $p'$-part of $|H\cap X|$, and so we conclude that $|X|<|H\cap X|^3$. The candidates for such maximal subgroups $H$ are recorded in \cite{a:AB-Large-15}. In this list, we only need to consider those subgroups $H$ which belongs to $\Cmc_{i}$ for $i=1,2,3,5,8$, see Proposition \ref{prop:class-s}. We then analyse each possible cases in Propositions \ref{prop:psl}-\ref{prop:orth}. Note that the case where $X=\PSp_4(q)$ needs more attention, see Proposition \ref{prop:psp}. If $X$ is a finite exceptional simple group in characteristic $p$, then we show that $r\neq p$, and since the subgroup $H$ is large in $G$, that is to say, $|G|\leq |H|^3$, using the list of such subgroups in \cite{a:ABD-Exp}, we obtain no possible parameters satisfying hypothesises of our main result. In this paper, we use the software \textsf{GAP} \cite{GAP4} for computational arguments.

\section{Preliminaries}\label{sec:pre}

In this section, we state some useful facts in both design theory and group theory. Lemma \ref{lem:New} below is an elementary result on subgroups of almost simple groups.

\begin{lemma}\label{lem:New}{\rm \cite[Lemma 2.2]{a:ABD-PSL3}}
Let $G$  be an almost simple group with socle $X$, and let $H$ be maximal in $G$ not containing $X$. Then $G=HX$ and $|H|$ divides $|\Out(X)|{\cdot}|H\cap X|$.
\end{lemma}

\begin{lemma}\label{lem:Tits}{\rm (Tits' Lemma~\cite[1.6]{a:tits})}
If $X$ is a simple group of Lie type in characteristic $p$, then any proper subgroup of index prime to $p$ is contained in a parabolic subgroup
of $X$.
\end{lemma}

%\begin{lemma}[]\label{lem:proper}{\rm~\cite[Lemma 7]{a:reg-classical}}
%If $X$ is a simple group of Lie type in characteristic $2$ except for  $\PSL_{2}(4)\cong \A_{5}$ and $\PSp_{4}(2)'\cong \A_{6}$, then
%any proper subgroup $H$ such that $|X{:}H|_{2}\leq 2$ is contained in a parabolic subgroup of $X$.
%\end{lemma}

If a group $G$ acts on a set $\Pmc$ and $\alpha\in \Pmc$, the \emph{subdegrees} of $G$ are the size of orbits of the action of the point-stabiliser $G_\alpha$ on $\Pmc$.

\begin{lemma}\label{lem:subdeg}{\rm \cite[3.9]{a:LSS1987}}
If $X$ is a group of Lie type in characteristic $p$, acting on the set of cosets of a maximal parabolic subgroup, and $X$ is neither $\PSL_n(q)$, $\POm_{n}^{+}(q)$
(with $n/2$ odd), nor $E_{6}(q)$, then there is a unique subdegree which is a power of $p$.
\end{lemma}

\begin{remark}\label{rem:subdeg}
We remark that even in the cases excluded in Lemma~\ref{lem:subdeg}, many of the maximal parabolic subgroups still have the property as asserted, see proof of  \cite[Lemma 2.6]{a:Saxl2002}. In particular, for an almost simple group $G$ with socle $X=E_{6}(q)$, if $G$  contains a graph automorphism or $H =P_{i}$ with $i$ one of $2$ and $4$, the conclusion of Lemma~\ref{lem:subdeg} is still true.
\end{remark}

\begin{lemma}\label{lem:six} {\rm \cite[Lemmas 5 and 6]{a:Zhou-nonsym-alt}}
Let $\Dmc$ be a $2$-design with prime replication number $r$, and let $G$ be a flag-transitive automorphism group of $\Dmc$. If $\alpha$ is a point in $\Pmc$ and $H:=G_{\alpha}$, then
\begin{enumerate}[\rm \quad (a)]
\item $r(k-1)=\lambda(v-1)$.  In particular, if $r$ is prime, then $r$ divides $v-1$ and $\gcd(r,v)=1$;
\item $vr=bk$;
%\item $bk(k-1)=\lambda v(v-1)$;
\item $r\mid |H|$ and $\lambda v<r^2$;
%\item $r\mid \lambda \gcd(v-1,|H|)$;
\item $r\mid d$, for all nontrivial subdegrees $d$ of $G$.
\end{enumerate}
\end{lemma}

For a given positive integer $n$ and a prime divisor $p$ of $n$, we denote the $p$-part of $n$ by $n_{p}$, that is to say, $n_{p}=p^{t}$ with $p^{t}\mid n$ but $p^{t+1}\nmid n$.

\begin{corollary}\label{cor:large}
Let $\Dmc$ be a flag-transitive $2$-design with automorphism group $G$. Then $|G|\leq |G_{\alpha}|^3$,  where $\alpha$ is a point in $\Dmc$.
\end{corollary}
\begin{proof}
By Lemma \ref{lem:six}(c), we have that $v<r^2$. The result follows from th  $v{=}|G{:}G_{\alpha}|$ and $r\leq |G_{\alpha}|$.
\end{proof}

%\begin{corollary}\label{cor:p-part}
%Suppose that $\Dmc$ is a a $2$-design with prime replication number $r$. Let $G$ be a primitive, flag-transitive almost simple automorphism group of $\Dmc$ with simple socle $X$ of Lie type in characteristic $p$. If the point-stabiliser $G_{\alpha}$ is not a parabolic subgroup of $G$, Then $|G|<|G_{\alpha}|{\cdot}|G_{\alpha}|_{p'}^2$.
%\end{corollary}
%\begin{proof}
%It follows from Corollary~\ref{cor:large} that $|G|<|G_{\alpha}|^3$. By~\cite[1.6]{a:tits}, $p$ divides $v= [G:G_{\alpha}]$ and since $r$ divides $\lambda(v-1)$ with $r$ prime, the parameter $r$ divides $|G_{\alpha}|_{p'}$, and since $v<r^2$, we have $|G|< |G_{\alpha}|{\cdot}|G_{\alpha}|_{p'}^2$.
%\end{proof}

\begin{proposition}\label{prop:flag}
Let $\Dmc$ be a $2$-design with prime replication number $r$ admitting a flag-transitive automorphism group $G$. Then $G$ is point-primitive of almost simple or affine type.
\end{proposition}
\begin{proof}
Since the parameter $r$ is prime, it follows that $r$ is coprime to $\lambda$, and so by \cite[2.3.7(a)]{b:dembowski}, we conclude that $G$ is point-primitive. Moreover,
\cite[Theorem]{a:ZIESCHANG} implies that $G$ is of almost simple or affine type.
\end{proof}

\begin{lemma}\label{lem:divisible}
Suppose that $\Dmc$ is a $2$-design with replication number $r$. Let $G$ be a  flag-transitive automorphism group of $\Dmc$ with simple socle $X$ of Lie type in characteristic $p$. If the point-stabiliser $H=G_{\alpha}$ contains a normal quasi-simple subgroup $K$ of Lie type in characteristic $p$ and $p$ does not divide $|Z(K)|$, then either $p$ divides $r$, or $K_{B}$ is contained in a parabolic subgroup $P$ of $K$ and $r$ is divisible by $|K{:}P|$.
\end{lemma}
\begin{proof}
If $B$ is a block incident with a point $\alpha$ of $\Dmc$, then $r= |H{:}H_{B}|$, and so $|K{:}K_{B}|$ divides $r$. If $\gcd(r,p)=1$, then $|K{:}K_{B}|$ is coprime to $p$, and now Lemma~\ref{lem:Tits} implies that $K_{B}$ is contained in a (maximal) parabolic subgroup $P$ of $K$. Hence $r$ is divisible by $|K{:}P|$.
\end{proof}

\section{Proof of the main result}\label{sec:proof}

In this section, we prove Theorem \ref{thm:main}. suppose that $\Dmc$ is a nontrivial $2$-design with prime replication number $r$ and $G$ is an almost simple automorphism group of $\Dmc$ with socle $X$ being a finite non-abelian simple group. Suppose now that $G$ is flag-transitive. Then Proposition~\ref{prop:flag} implies that $G$ is point-primitive. Let $H:=G_{\alpha}$, where $\alpha$ is a point of $\Dmc$. Therefore, $H$ is maximal in $G$ (see \cite[Corollary 1.5A]{b:Dixon}), and so Lemma~\ref{lem:New} implies that
\begin{align}
v=\frac{|X|}{|H\cap X|}.\label{eq:v}
\end{align}
In the proceeding  sections, we consider each possibility for the non-abelian simple group $X$.

\subsection{Alternating and sporadic groups}\label{sec:alt}

Here, we consider the case where $X$ is an alternating group $\A_{n}$ ($n\geq 5$) or a sporadic simple group, and note that the result follows immediately from the main results in \cite{a:Zhou-lam-large-sporadic,a:Zhan-nonsym-sprodic,a:Zhou-nonsym-alt,a:Zhu-sym-alternating}:

\begin{proposition}\label{prop:alt-spor}{\rm \cite{a:Zhou-lam-large-sporadic,a:Zhan-nonsym-sprodic,a:Zhou-nonsym-alt,a:Zhu-sym-alternating}}
Let $\Dmc$ be a nontrivial $2$-design with prime replication number $r$. Suppose that $G$ is an automorphism group of $\Dmc$ of almost simple type with socle $X$ a sporadic simple group or an alternating group $\A_{n}$ with $n\geq 5$. If $G$ is flag-transitive, then $(v,b,r,k,\lambda)$, $G$, $X$ and $G_{\alpha}\cap X$ are as in one of the lines $2$ and $5$-$8$ of {\rm Table~\ref{tbl:main}}.
\end{proposition}

\subsection{Classical groups.}\label{sec:classical}

In this section, we suppose that $G$ is an almost simple group with socle $X=X(q)$ being a finite classical simple group of Lie type, where $q=p^{a}$ for some positive integer $a$. We here use the following notation to denote the finite classical simple groups.
\begin{align*}
&\PSL_{n}(q), \text{ for } n\geq 2 \text{ and } (n,q)\neq (2,2), (2,3),\\
&\PSU_{n}(q), \text{ for } n\geq 3 \text{ and } (n,q)\neq (3,2),\\
&\PSp_{2n}(q), \text{ for } n\geq 2 \text{ and } (n,q)\neq (4,2),\\
&\POm_{2n+1}(q), \text{ for } n\geq 3 \text{ and } q \text{ odd },\\
&\POm_{2n}^{\pm}(q),  \text{ for } n\geq 4.
\end{align*}
In this manner, the only repetitions are:
\begin{align}\label{eq:iso}
\nonumber  &\PSL_{2}(4)\cong \PSL_{2}(5)\cong \A_{5}, \quad
\PSL_{2}(7)\cong \PSL_{3}(2), \quad
\PSL_{2}(9)\cong \A_{6},\\
&\PSL_{4}(2)\cong \A_{8},\\
\nonumber  &\PSp_{4}(3)\cong \PSU_{4}(2).
\end{align}

For a finite group $X$, let $p(X)$ be the minimal degree of permutation representation of $X$. In particular, for a finite simple group $X$, the integer $p(X)$ is just the index of the largest proper subgroup of $X$, and we know these degrees for all finite simple groups. Here, we need $p(X)$ for finite classical simple groups $X$:

\begin{lemma}\label{lem:min-deg}{\rm \cite{a:Cooperstein-1}}
The minimal degrees $p(X)$ of permutation representations of finite classical simple groups of Lie type are given in {\rm Table~\ref{tbl:min-deg}}.
\end{lemma}
\begin{table}%[h]
\centering
\scriptsize
\caption{The minimal degrees of permutation representations of finite  classical  simple groups of Lie type.}\label{tbl:min-deg}
\resizebox{\textwidth}{!}{
\begin{tabular}{lll}
\hline
$X$ & $p(X)$ & Conditions\\\hline
$\PSL_{n}(q)$ & $(q^n-1)/(q-1)$ & $(n, q)\neq (2, 5), (2, 7), (2, 9), (2, 11), (4, 2)$\\
$\PSL_{2}(q)$ & $ q $ & $q=5, 7, 11$\\
$\PSL_{2}(9)$ & $ 6 $ & \\
$\PSL_{4}(2)$ & $8$ & \\
$\PSU_{n}(q)$ & $(q^{n}-(-1)^{n})(q^{n-1}-(-1)^{n-1})/(q^2-1)$ & $n\geq 5$ and $(n, q)\neq (6s, 2)$\\
$\PSU_{n}(2)$ & $2^{n-1}(2^n-1)/3$ & $n\equiv 0 \mod{6}$ \\
$\PSU_{4}(q)$ & $(q+1)(q^3+1)$ & \\
$\PSU_{3}(q)$ & $q^3+1$ & $q\neq 5$\\
$\PSU_{3}(5)$ & $50$ & \\
$\PSp_{2n}(q)$ & $(q^{2n}-1)/(q-1)$ & $n\geq 2$, $q> 2$ and $(n, q)\neq (2, 3)$\\
$\PSp_{2n}(2)$ & $2^{n-1}(2^{n}-1)$ & $n\geq 3$\\
$\PSp_{4}(3)$ & $27$ & \\
$\POm_{2n+1}(q)$ & $(q^{2n}-1)/(q-1)$ & $n\geq 3$, $q$ odd and $q\geq 5$\\
$\POm_{2n+1}(3)$ & $3^{n}(3^n-1)/2$ & $n\geq 3$\\
$\POm_{2n}^{+}(q)$ & $(q^{n}-1)(q^{n-1}+1)/(q-1)$ & $n\geq 4$ and $q\geq 3$\\
$\POm_{2n}^{+}(2)$ & $2^{n-1}(2^{n}-1)$ & $n\geq 4$\\
$\POm_{2n}^{-}(q)$ & $(q^{n}+1)(q^{n-1}-1)/(q-1)$ & $n\geq 4$\\
\hline
%\multicolumn{10}{p{8.7cm}}{\tiny $\ast$ The last column addresses to references in which a design with the parameters in the line has been constructed.}
\end{tabular}
}
\end{table}

\begin{lemma}\label{lem:coprime}
Let $\Dmc$ be a $2$-design with prime replication number $r$, and let $G$ be a flag-transitive automorphism group of $\Dmc$ of almost simple type with socle $X$ being a finite  classical simple  group in characteristic $p$. Let also $\alpha$ be a point of $\Dmc$ and $H=G_{\alpha}$. If $X\neq \PSL_{2}(q)$, then $r$ divides $|H\cap X|_{p'}$, and hence  $|X|<|H\cap X|{\cdot}|H\cap X|_{p'}^2$.
\end{lemma}
\begin{proof}
By Lemma~\ref{lem:six}(c), the parameter $r$ divides $|H|$, and so Lemma~\ref{lem:New} implies that $r$ divides $|\Out(X)|{\cdot}|H\cap X|$. We show that $r$ is coprime to $|\Out(X)|{\cdot}|H\cap X|_{p}$, and hence we conclude that $r$ divides $|H\cap X|_{p'}$, as claimed.

In what follows, we assume that $X\neq \PSL_{2}(q)$. Further, assume the contrary that $r$ divides $|\Out(X)|{\cdot}|H\cap  X|_{p}$. Then $r$ divides $|\Out(X)|{\cdot}|X|_{p}$. We now consider each possibility for $X$ separately.\smallskip

Suppose that $X=\PSL_{n}(q)$ with $n\geq 3$ and $q=p^{a}$. If $(n, q)\neq (4, 2)$, then by Lemma~\ref{lem:min-deg}, we have that $v\geq p(X)=(q^{n}-1)/(q-1)$, and hence Lemma~\ref{lem:six}(c) implies that
\begin{align}\label{eq:psl-1}
r^{2}> \frac{q^{n}-1}{q-1}.
\end{align}
Since $r$ is a prime number dividing $|\Out(X)|{\cdot}|X|_{p}$, by inspection $|\Out(X)|$ and $|X|_{p}$ from \cite[Table 5.1.A]{b:KL-90}, we conclude that $r$ divides $2ap\cdot \gcd(n,q-1)$. Then $r\in \{2,p\}$ or $r$ divides $a$ or $q-1$, and so $r\leq \max\{a,p,q-1\}\leq q$. Since $n\geq 3$, we conclude  by \eqref{eq:psl-1} that $q^{2}+q+1\leq (q^{n}-1)/(q-1)<r^{2}\leq q^{2}$, and so $q^{2}+q+1<q^{2}$, which is a contradiction. If $X=\PSL_{4}(2)$, then $r$ divides $|\Out(X)|{\cdot} |X|_{2}=2^{7}$, and so $r=2$. By Lemmas~\ref{lem:min-deg} and~\ref{lem:six}(c), we have that $8\leq v<r^{2}=4$, which is impossible.\smallskip

Suppose that $X=\PSU_{n}(q)$ with $n\geq 3$ and $q=p^{a}$. Assume first that $n\geq 5$. If $(n, q)\neq (6m, 2)$, then by Lemma~\ref{lem:min-deg}, we have that $v\geq p(X)=(q^{n}-(-1)^{n})(q^{n-1}-(-1)^{n-1})/(q^2-1)$, and hence Lemma~\ref{lem:six}(c) implies that
\begin{align}\label{eq:psu-1}
r^{2}> \frac{(q^{n}-(-1)^{n})(q^{n-1}-(-1)^{n-1})}{q^2-1}.
\end{align}
Since $r$ is a prime number dividing $|\Out(X)|{\cdot}|X|_{p}$, by inspection $|\Out(X)|$ and $|X|_{p}$ from \cite[Table 5.1.A]{b:KL-90}, the parameter $r$ divides $2ap\cdot \gcd(n, q+1)$. Then $r\in \{2, p\}$ or $r$ divides $a$ or $q+1$, and so $r\leq \max\{a, p, q+1\}\leq q+1$. Note that $n\geq 5$. Then it follows from   \eqref{eq:psu-1} that $(q^{5}+1)(q^2+1)\leq (q^{n}-(-1)^{n})(q^{n-1}-(-1)^{n-1})/(q^2-1)<r^{2}\leq (q+1)^{2}$, that is to say, $(q^{5}+1)(q^2+1)<(q+1)^{2}$, which is a contradiction. If $(n, q)=(6m, 2)$, then $r$ is a prime divisor of $|\Out(X)|{\cdot}|X|_{2}$ dividing $2\cdot 3\cdot 2^{n(n-1)/2}$, and so $r=2$ or $3$, by Lemmas~\ref{lem:six}(c) and~\ref{lem:min-deg}, we must have  $2^{n-1}(2^n-1)<3r^{2}\leq 3^{3}$, for some $n\geq 6$,  which is impossible. Assume now that $n= 4$. Since $r$ divides $|\Out(X)|{\cdot}|X|_{p}$, it follows from  \cite[Table 5.1.A]{b:KL-90} that $r$ divides $8ap$. Then $r\in \{2, p\}$ or $r$ divides $a$, and so $r\leq \max\{a, p\}\leq q$. By Lemmas~\ref{lem:six}(c) and~\ref{lem:min-deg}, we have that $(q+1)(q^3+1)\leq \lambda v< r^{2}\leq q^2$, and so $(q+1)(q^3+1)<q^2$, which is impossible. Finally assume that $n= 3$. If $q\neq 5$, then by Lemma~\ref{lem:min-deg}, we have that $v\geq p(X)=q^3+1$. We use the information in \cite[Table 5.1.A]{b:KL-90} and conclude that $r$ divides $6ap$. Then $r\in \{2, 3, p\}$ or $r$ divides $a$, and so $r\leq \max\{a, 3, p\}\leq q$. By Lemma~\ref{lem:six}(c) we conclude that $q^3+1\leq \lambda v< r^{2}\leq q^2$, and so $q^3+1<q^2$, which is impossible. If $q=5$, then from \cite[Table 5.1.A]{b:KL-90}, we conclude that $r$ divides $30$. Then $r\in \{2, 3, 5\}$, and so $r\leq 5$. By Lemmas~\ref{lem:six}(c) and~\ref{lem:min-deg}, we have that $50\leq \lambda v< r^{2}\leq 25$, which is a contradiction.\smallskip

Suppose that $X=\PSp_{2n}(q)$ with $n\geq 2$ and $q=p^{a}$. Assume first that $q> 2$ and $(n , q)\neq(2, 3)$. Then by Lemma~\ref{lem:min-deg}, we have that $v\geq p(X)=(q^{2n}-1)/(q-1)$, and hence Lemma~\ref{lem:six}(c) implies that
\begin{align}\label{eq:psp-1}
r^{2}> \frac{q^{2n}-1}{q-1}.
\end{align}
Since $r$ is a prime number dividing $|\Out(X)|{\cdot}|X|_{p}$, by inspection $|\Out(X)|$ and $|X|_{p}$ from \cite[Table 5.1.A]{b:KL-90}, the parameter $r$ must divide $2ap$. Then $r\in \{2,p\}$ or $r$ divides $a$, and so $r\leq \max\{a, p\}\leq q$. Since $n\geq 2$, it follows from \eqref{eq:psp-1} that $(q^{2}+1)(q+1)\leq (q^{2n}-1)/(q-1)<r^{2}\leq q^2$, and so $(q^{2}+1)(q+1)<q^{2}$, which is a contradiction. Assume now that $n\geq 3$ and $q=2$. In this case, $r$ divides $|\Out(X)|{\cdot}|X|_{2}$, and so by \cite[Table 5.1.A]{b:KL-90}, we must have $r=2$, which is a contradiction. Assume finally that $(n, q)=(2, 3)$. Then $r$ divides $|\Out(X)|{\cdot}|X|_{3}=6$, and so $r=2$ or $3$. By Lemmas~\ref{lem:min-deg} and~\ref{lem:six}(c), we have that $27\leq v<r^{2}\leq 9$, which is impossible.\smallskip

Suppose that $X=\POm_{2n+1}(q)$ with $n\geq 3$ and $q=p^{a}$ odd. Assume first that $q\neq 3$. Since $r$ is a prime number dividing $|\Out(X)|{\cdot}|X|_{p}$, by inspection $|\Out(X)|$ and $|X|_{p}$ from \cite[Table 5.1.A]{b:KL-90}, we conclude that $r$ divides $2ap$. Then $r\in \{2,p\}$ or $r$ divides $a$, and so $r\leq \max\{a, p\}\leq q$. Since $n\geq 3$, we conclude by Lemmas~\ref{lem:min-deg} and~\ref{lem:six}(c) that $q^{5}+q^{4}+q^{3}+q^{2}+q+1\leq (q^{2n}-1)/(q-1)<r^{2}\leq q^2$, and so $q^{5}+q^{4}+q^{3}+q^{2}+q+1<q^{2}$, which is a contradiction.\smallskip
Assume finally that $q=3$ and $n\geq 3$.  Then $r$ divides $|\Out(X)|{\cdot}|X|_{3}$, and so $r=2$ or $3$. By Lemmas~\ref{lem:min-deg} and~\ref{lem:six}(c), we have that $3^{n}(3^{n-1}-1)<2r^{2}\leq 18$, for some positive integer $n\geq 3$, which is impossible.\smallskip

Suppose that $X=\POm_{2n}^{\e}(q)$ with $n\geq 4$, $q=p^{a}$ and $\e \in \{+, -\}$. Assume first that $(\e, q)\neq (+, 2)$. Since $r$ is a prime divisor of $|\Out(X)|{\cdot}|X|_{p}$, it follows from \cite[Table 5.1.A]{b:KL-90} that $r$ divides $6ap$, and so $r\in \{2, 3, p\}$ or $r$ divides $a$. Then $r\leq \max\{a, p\}\leq q$. Since $n\geq 4$, we conclude by Lemmas~\ref{lem:min-deg} and~\ref{lem:six}(c) that $(q^{4}-\e1)(q^{3}+\e1)\leq (q^{n}-\e1)(q^{n-1}+\e1)\leq \lambda (q-1)v< (q-1)r^{2}\leq q^2(q-1)$, and so $(q^{4}-\e1)(q^{3}+\e1)< q^2(q-1)$, which is a contradiction. Assume finally that $(\e, q)= (+, 2)$.  Then it follows from \cite[Table 5.1.A]{b:KL-90}  that $r$ divides $6$, and so $r=2$ or $3$. By Lemmas~\ref{lem:min-deg} and~\ref{lem:six}(c), we have that $2^{n-1}(2^{n}-1)<r^{2}\leq 9$, for some positive integer $n\geq 4$, which is impossible.

Therefore, the parameter $r$ divides $|H\cap  X|_{p'}$. Note by Proposition~\ref{prop:flag} that the group $G$ is primitive, and by the assumption it is almost simple with socle $X$. Then $H=G_{\alpha}$ is maximal in $G$, and so by Lemma~\ref{lem:New}, we have that $v=|X{:}H\cap X|$. Therefore, Lemma~\ref{lem:six}(c) implies that $|X:H\cap X|=v< r^{2}\leq |H{\cap }X|_{p'}^{2}$, and hence $|X|<|H\cap  X|{\cdot}|H\cap X|_{p'}^2$.
\end{proof}

Suppose now that $\Dmc$ is a nontrivial $2$-design with prime replication number $r$. Moreover, suppose that $G$ is a flag-transitive automorphism group of $\Dmc$ with  socle $X$. Recall by Proposition~\ref{prop:flag} that the group $G$ acts primitively on the point set of $\Dmc$, and so the point-stabiliser $H$ is maximal in $G$. We now apply Aschbacher's Theorem~\cite{a:Aschbacher} which says that $H$ lies in one of the eight geometric families $\Cmc_{i}$ of subgroups of $G$, or in the family $\Smc$ of almost simple subgroups with some irreducibility conditions. We follow the description of these subgroups as in \cite{b:KL-90} and analyse each of these cases separately. In what follows, if $H$ belongs to the family $\Cmc_{i}$, for some $i$, then we sometimes say that $H$ is a $\Cmc_{i}$-subgroup. We also denote by $\,^{\hat{}}H$ the pre-image of the group $H$ in the corresponding linear group.  A rough description of the $\Cmc_i$ families is given in Table \ref{t:max}.

\begin{table}%[h]
%\centering
\scriptsize
\caption{The geometric subgroup collections}\label{t:max}
\begin{tabular}{clllll}
\hline
Class & Rough description\\
\hline
$\Cmc_1$ & Stabilisers of subspaces of $V$\\
$\Cmc_2$ & Stabilisers of decompositions $V=\bigoplus_{i=1}^{t}V_i$, where $\dim V_i  = a$\\
$\Cmc_3$ & Stabilisers of prime index extension fields of $\F$\\
$\Cmc_4$ & Stabilisers of decompositions $V=V_1 \otimes V_2$\\
$\Cmc_5$ & Stabilisers of prime index subfields of $\F$\\
$\Cmc_6$ & Normalisers of symplectic-type $r$-groups in absolutely irreducible representations\\
$\Cmc_7$ & Stabilisers of decompositions $V=\bigotimes_{i=1}^{t}V_i$, where $\dim V_i  = a$\\
$\Cmc_8$ & Stabilisers of non-degenerate forms on $V$\\ \hline
\end{tabular}
\end{table}

\begin{table}%[h]
%\centering
\scriptsize
\caption{Some large maximal subgroups of finite classical simple groups.}\label{tbl:class-s}
\resizebox{\textwidth}{!}{
\begin{tabular}{clllll}
\hline
Class &
$X$ &
$H\cap X$ &
$v$ &
$u_r$ &
Conditions \\
\hline
$\Cmc_{6}$ &
$\PSL_3(4)$ &
$3^2.\Q_8$ &
$2^{3}{\cdot} 5{\cdot} 7$ &
$3$
\\
$\Smc$ &
$\PSL_{3}(4)$ &
$\A_{6}$ &
$2^{3}{\cdot} 7$ &
$5$  &
\\
$\Smc$ &
$\PSL_{4}(2)$ &
$\A_{7}$ &
$2^{3}$ &
$7$  &
\\
$\Cmc_{6}$ &
$\PSL_4(5)$ &
$2^4.\A_{6}$ &
$5^{5}{\cdot} 13{\cdot} 31$ &
$5$
\\
$\Smc$ &
$\PSL_{4}(7)$ &
$\PSU_{4}(2)$ &
$2^{3}{\cdot} 5{\cdot} 7^{6}{\cdot} 19$ &
$5$  &
\\
$\Smc$ &
$\PSL_{5}(3)$ &
$\M_{11}$ &
$2^{5}{\cdot} 3^{8}{\cdot} 11{\cdot} 13$ &
$11$  &
\\
$\Smc$ &
$\PSU_{3}(3)$ &
$\PSL_{2}(7)$ &
$2^{2}{\cdot} 3^{2}$ &
$7$ &
\\
$\Smc$ &
$\PSU_{3}(5)$&
$\PSL_{2}(7)$ &
$2{\cdot} 3{\cdot} 5^{3}$ &
$7$ &
\\
$\Smc$ &
$\PSU_{3}(5)$&
$\A_{7}$ &
$2{\cdot} 5^{2}$ &
$7$ &
\\
$\Smc$ &
$\PSU_{3}(5)$&
$\M_{10}$ &
$5^{2}{\cdot} 7$ &
$5$ &
\\
$\Cmc_{6}$ &
$\PSU_{4}(3)$ &
$2^{4}.\A_{6}$ &
$3^{4}{\cdot} 7$ &
$5$
\\
$\Smc$ &
$\PSU_{4}(3)$ &
$\PSL_{3}(4)$ &
$2{\cdot} 3^{4}$ &
$7$ &
\\
$\Smc$ &
$\PSU_{4}(3)$&
$\A_{7}$ &
$2^{4}{\cdot} 3^{4}$ &
$7$ &
\\
$\Smc$ &
$\PSU_{4}(5)$&
$\A_{7}$ &
$2^{4}{\cdot} 3^{2}{\cdot} 5^{5}{\cdot} 13$ &
$7$ &
\\
$\Smc$ &
$\PSU_{4}(5)$ &
$\PSU_{4}(2)$  &
$2{\cdot} 5^{5}{\cdot} 7{\cdot} 13$ &
$5$ &
\\
$\Cmc_{6}$ &
$\PSU_{4}(7)$ &
$2^{4}.\Sp_{4}(2)$ &
$2^{2}{\cdot} 5{\cdot} 7^{6}{\cdot} 43$ &
$5$
\\
$\Smc$ &
$\PSU_{5}(2)$ &
$\PSL_{2}(11)$  &
$2^{8}{\cdot} 3^{4}$ &
$11$ &
\\
$\Smc$ &
$\PSU_{6}(2)$ &
$\M_{22}$&
$2^{8}{\cdot} 3^{4}$ &
$11$ &
\\
$\Smc$ &
$\PSU_{6}(2)$&
$\PSU_{4}(3).2$  &
$2^{9}{\cdot} 3^{3}{\cdot} 5 {\cdot} 11$ &
$7$ &
\\
$\Smc$ &
$\PSp_{4}(q)$ &
$\Sz(q)$ &
$q^2(q^2-1)(q+1)$&
$q^2+1$&
$q=2^{a}\geq 4$\\
$\Smc$ &
$\PSp_{4}(7)$ &
$\A_{7}$ &
$2^{5} {\cdot} 5{\cdot} 7^{3}$ &
$7$ &
\\
$\Smc$ &
$\PSp_{4}(5)$ &
$\A_{6}$ &
$2^{3}{\cdot} 5^{3}{\cdot} 13$ &
$5$ &
\\
$\Smc$ &
$\PSp_{4}(2)$ &
$\A_{5}$ &
$2^{2}{\cdot} 3$ &
$5$ &
\\
$\Cmc_{6}$ &
$\PSp_{4}(3)$ &
$2^4.\Omega_{4}^{-}(2)$ &
$3^{3}$ &
$5$
\\
$\Cmc_{6}$ &
$\PSp_{4}(5)$&
$2^{4}.\Omega_{4}^{-}(2)$ &
$3{\cdot} 5^{3}{\cdot} 13$ &
$5$
\\
$\Cmc_{6}$ &
$\PSp_{4}(7)$ &
$ 2^4.\O_{4}^{-}(2)$ &
$2{\cdot} 3{\cdot} 5{\cdot} 7^{4}$ &
$5$
\\
$\Smc$ &
$\PSp_{6}(q)$ &
$\G_{2}(q)$ &
$q^3(q^4-1)$&
$q^2+q+1$ &
$q$ even\\
$\Smc$ &
$\PSp_{6}(2)$&
$\PSU_{3}(3).2$ &
$2^{3}{\cdot} 3{\cdot} 5$ &
$7$&
\\
$\Smc$ &
$\PSp_{6}(5)$&
$\mathrm{J}_{2}$  &
$2^{2}{\cdot} 3{\cdot} 5^{7}{\cdot} 13{\cdot} 31$ &
$7$&
\\
$\Smc$ &
$\PSp_{8}(2)$ &
$\S_{10}$ &
$2^{8}{\cdot} 3{\cdot} 17$ &
$7$ &
\\
$\Cmc_{6}$ &
$\PSp_{8}(3)$&
$2^{6}.\Omega_{6}^{-}(2)$ &
$2^{2}{\cdot} 3^{12}{\cdot} 5{\cdot} 7{\cdot} 13{\cdot} 41$ &
$5$
\\
$\Smc$ &
$\PSp_{12}(2)$ &
$\S_{14}$ &
$2^{25}{\cdot} 3^{3}{\cdot} 5{\cdot} 17 {\cdot} 31$ &
$13$ &
\\
$\Smc$ &
$\PSp_{16}(2)$ &
$\S_{18}$ &
$2^{48}{\cdot} 3^{2}{\cdot} 5{\cdot} 17{\cdot} 31{\cdot} 43{\cdot} 127{\cdot} 257$ &
$17$ &
\\
$\Smc$ &
$\PSp_{20}(2)$ &
$\S_{22}$ &
$2^{81}{\cdot} 3^{5}{\cdot} 5^{2}{\cdot} 17{\cdot} 31^{2}{\cdot} 41{\cdot} 43{\cdot} 73{\cdot} 127{\cdot} 257$ &
$19$ &
\\
$\Smc$ &
$\POm_{7}(q)$ &
$\G_{2}(q)$  &
$q^3(q^4-1)/2$&
$q^2+q+1$&
\\
$\Smc$ &
$\POm_{7}(3)$ &
$\Sp_{6}(2)$&
$3^{5}{\cdot} 13$ &
$7$ &
\\
$\Smc$ &
$\POm_{7}(3)$&
$\S_{9}$  &
$2^{2}{\cdot} 3^{5}{\cdot} 13$ &
$7$ &
\\
$\Smc$ &
$\POm_{7}(5)$ &
$\Sp_{6}(2)$ &
$5^{8}{\cdot} 13{\cdot} 31$ &
$7$&
\\
$\Smc$ &
$\POm_{7}(7)$ &
$\Sp_{6}(2)$ &
$2^{3}{\cdot} 5{\cdot} 7^{8}{\cdot} 19{\cdot} 43$ &
$7$&
\\
$\Smc$ &
$\POm_{9}(3)$ &
$\A_{10}$ &
$2^{7}{\cdot} 3^{12}{\cdot} 13{\cdot} 41$&
$7$&
\\
$\Smc$ &
$\POm_{8}^{+}(q)$ &
$\Omega_{7}(q)$ &
$q^3(q^4-1)/2$ &
$q^2+q+1$&
$q$ odd\\
$\Cmc_{4}$ &
$\POm_{8}^{+}(q)$ &
$\PSp_{4}(q){\times}\PSp_{2}(q)$ &
$q^7(q^6-1)(q^2+1)$ &
$q^2+1$&
\\
$\Smc$ &
$\POm_{8}^{+}(q)$&
$\PSp_{6}(q)$ &
$q^3(q^4-1)$&
$q^2+q+1$&
$q$ even\\
$\Smc$ &
$\POm_{8}^{+}(q)$&
${}^3\D_{4}(q_{0})$ &
$q_{0}^{24}(q_{0}^{18}-1)(q_{0}^{12}-1)(q_{0}^{6}-1)(q_{0}^{2}+1)$&
$q_{0}^{8}+q_{0}^4+1$&
$q=q_{0}^3$ odd\\
$\Smc$ &
$\POm_{8}^{+}(q)$&
$\POm_{8}^{-}(q_{0})$ &
$q^{6}(q^{6}-1)(q^{3}+1)(q+1)$&
$q^{2}+1$&
$q=q_{0}^2$\\
$\Smc$ &
$\POm_{8}^{+}(2)$  &
$\A_{9}$ &
$2^{6}{\cdot} 3{\cdot} 5$ &
$7$ &
\\
$\Smc$ &
$\POm_{8}^{+}(3)$&
$\Omega_{8}^{+}(2)$&
$3^{7}{\cdot} 13$&
$7$&
\\
$\Smc$ &
$\POm_{8}^{+}(5)$&
$\Omega_{8}^{+}(2)$&
$5^{10}{\cdot} 13^{2}{\cdot} 31$&
$7$&
\\
$\Smc$ &
$\POm_{8}^{+}(7)$&
$\Omega_{8}^{+}(2)$&
$2^{4}{\cdot} 5^{2}{\cdot} 7^{11}{\cdot} 19{\cdot} 43$&
$7$&
\\
$\Smc$ &
$\POm_{10}^{-}(2)$ &
$\M_{12}$ &
$2^{14}{\cdot} 3^{3}{\cdot} 5{\cdot} 7{\cdot} 17$ &
$11$ &
\\
$\Smc$ &
$\POm_{10}^{-}(2)$ &
$\A_{12}$ &
$2^{11}{\cdot} 3{\cdot} 17$&
$11$ &
\\
$\Smc$ &
$\POm_{10}^{+}(3)$ &
$\A_{12}$ &
$2^{6}{\cdot} 3^{15}{\cdot} 11{\cdot} 13{\cdot} 41$&
$11$ &
\\
$\Cmc_{4}$ &
$\POm_{12}^{+}(q)$ &
$\PSp_{6}(q){\times}\PSp_{2}(q)$ &
$q^{20}(q^{10}-1)(q^{8}-1)(q^6-1)/(q^2-1)$ &
$q^3+1$&
\\
$\Smc$ &
$\POm_{12}^{-}(2)$ &
$\A_{13}$ &
$2^{21}{\cdot} 3{\cdot} 5{\cdot} 17{\cdot} 31$ &
$13$&
\\
$\Smc$ &
$\POm_{14}^{+}(2)$ &
$\A_{16}$ &
$2^{28}{\cdot} 3^{2}{\cdot} 17{\cdot} 31{\cdot} 127$ &
$13$&
\\
$\Smc$ &
$\POm_{16}^{+}(2)$ &
$\A_{17}$ &
$2^{42}{\cdot} 3^{4}{\cdot} 5{\cdot} 17{\cdot} 31{\cdot} 43{\cdot} 127$ &
$17$&
\\
$\Smc$ &
$\POm_{18}^{-}(2)$ &
$\A_{20}$ &
$2^{55}{\cdot} 3^{5}{\cdot} 17{\cdot} 31{\cdot} 43{\cdot} 127{\cdot} 257$ &
$19$&
\\
$\Smc$ &
$\POm_{20}^{-}(2)$ &
$\A_{21}$ &
$2^{73}{\cdot} 3^{4}{\cdot} 5^{2}{\cdot} 17{\cdot} 31{\cdot} 41{\cdot} 43{\cdot} 73{\cdot} 127{\cdot} 257$ &
$19$&
\\
$\Smc$ &
$\POm_{22}^{+}(2)$ &
$\A_{24}$ &
$2^{89}{{\cdot}}3^4{{\cdot}}5^2{{\cdot}}17{{\cdot}}31^2{{\cdot}}41{{\cdot}}43{{\cdot}}73{{\cdot}}89{{\cdot}}127{{\cdot}}257$ &
$23$&
\\
\hline
%\scriptsize
%Note: & \multicolumn{4}{p{10cm}}{\scriptsize The value $u_{r}$ is an upper bound for the parameter $r$ in Proposition~\ref{prop:class-s}.}\\
\end{tabular}
}
\end{table}

\begin{proposition}\label{prop:class-s}
Let $\Dmc$ be a nontrivial $2$-design with prime replication number $r$, and let $\alpha$ be a point of $\Dmc$. Suppose that $G$ is an automorphism group of $\Dmc$ of almost simple type with socle $X$ being a finite classical simple group of Lie type in characteristic $p$ and of dimension at least $3$. If $G$ is flag-transitive, then  $H=G_{\alpha}$ is maximal in $G$ and $H$ belongs to one of the geometric families  $\Cmc_{i}$, for $i=1,2,3,5,8$.
\end{proposition}
\begin{proof}
By Proposition~\ref{prop:flag}, the point-stabliliser $H$ is a maximal subgroup of $G$. Then by Lemma~\ref{lem:coprime}, $H$ is a large maximal subgroup of $G$ such that $|X|\leq |H\cap X|^{3}$. By \cite[Theorem 7]{a:AB-Large-15}, we conclude that $H$ does not belong to $\Cmc_{7}$. Assume now that $H$ belongs to $\Cmc_{6}$, $\Cmc_{4}$ or $\Smc$. We again apply \cite[Theorem 7]{a:AB-Large-15} and obtain the pairs $(X,H\cap X)$ listed in Table~\ref{tbl:class-s} can be read off from \cite[Tables 3 and 7]{a:AB-Large-15}. For each such $H\cap X$, by \eqref{eq:v}, we obtain $v$ as in the third column of Table~\ref{tbl:class-s}. Moreover, Lemma \ref{lem:coprime} says that $r$ divides $|H\cap X|_{p'}$, and so we can find an upper bound $u_{r}$ of $r$ as in the fourth column of Table~\ref{tbl:class-s}.  The inequality $\lambda v<r^{2}$ rules out all cases except for $(X,H\cap  X)$ being $(\PSL_{4}(2),\A_{7})$, $(\PSU_{3}(3),\PSL_{2}(7))$, or $(\PSp_{4}(2),\A_{5})$. If $H\cap  X\cong\A_{7}$ and $X=\PSL_{4}(2)$, then $v=8$, and since $r$ divides $v-1$ by Lemma~\ref{lem:six}(a), we conclude that $r=7$. Since also $bk=vr$, it follows that $b=14$ and $k=4$. This says that $G$  has a subgroup of index $14$, which is impossible. If $H\cap X\cong \PSL_{2}(7)$ and $X=\PSU_{3}(3)$, then $v=36$, and since $r$ divides both $v-1=35$ and $|H\cap X|_{7'}$ which is a divisor of $2^{4}\cdot 7$, it follows that $r=7$. We now apply Lemma~\ref{lem:six}(a)-(b) and conclude that $b=42$, $k=6$ and $\lambda=1$. This implies that $G=\PSU_{3}(3)$ has a subgroup of index $42$, which is a contradiction. If $H\cap X\cong \A_{5}$ and $X=\PSp_{4}(2)$, then $v=12$, and this case can also be ruled out as $r$ divides both $v-1=11$ and $|H\cap X|_{2'}$ which is a divisor of $3\cdot 5$.
\end{proof}

\begin{proposition}\label{prop:psl2}
Let $\Dmc$ be a nontrivial $2$-design with prime replication number $r$. Suppose that $G$ is an automorphism group of $\Dmc$ of almost simple type with socle $X=\PSL_{2}(q)$ for $q\geq 4$. If $G$ is flag-transitive, then {\rm Theorem \ref{thm:main}(a)} holds or $(v,b,r,k,\lambda)$, $G$, $X$ and $G_{\alpha}\cap X$ are as in line $1$,$3$ or $4$ of {\rm Table~\ref{tbl:main}}.
\end{proposition}
\begin{proof}
We can assume by \eqref{eq:iso} and Proposition~\ref{prop:alt-spor} that $q\geq 7$.
Note by Lemmas \ref{lem:New} and~\ref{lem:six}(c) that $r$ divides $|\Out(X)|{\cdot} |H\cap X|$. If $r$ divides $|\Out(X)|=\gcd(2,q-1){\cdot}a$, where $q=p^{a}$, then $r=2$ or $r$ divides $a$. We now apply Lemma~\ref{lem:min-deg} and by the same argument as in the proof of Lemma~\ref{lem:coprime}, we observe that $r$ cannot be a divisor of $|\Out(X)|$. Thus $r$ divides $|H\cap X|$, and hence \eqref{eq:v} and Lemma~\ref{lem:six}(c) implies that $|X|< |H\cap X|^{3}$. We are now ready to apply \cite[Theorem 7 and Proposition 4.7]{a:AB-Large-15} and \cite[Theorems 1.1 and 2.2]{a:gmax}, and conclude that $G$ and $H$ are as in one of the groups in Table~\ref{tbl:psl2} or $H\cap X$ is one of the following groups:
\begin{enumerate}[{\rm \quad(i)}]
\item $q{:}\frac{q-1}{\gcd(2, q-1)}$;
\item $\PSL(2, q_{0})$, for $q=q_{0}^3\geq 27$;
\item $\PGL(2, q_{0})$, for $ q= q_{0}^2\geq 9$;
\item $\D_{2(q-1)/\gcd(2, q -1)}$, if $q$ odd, then $q\geq 13$;
\item $\D_{2(q+1)/\gcd(2, q -1)}$, if $q$ odd, then $q\neq 7,9$;
\item $\A_{4}$, for $q\in\{5,13\}$;
\item $\S_{4}$, for $q\in\{7,17,23\}$;
\item $\A_{5}$, for $q\in\{9,11,19,29,31,41,49,59,61,71\}$.
\end{enumerate}
Since $r$ is a prime number dividing $\gcd(v-1,|H\cap X|)$, by examining the fact that $v<r^{2}$, we conclude that $H\cap X$ is isomorphic to one of the  groups $p {:} \frac{p-1}{\gcd(2,p-1)}$ for $r=p$, $\D_{2(q+1)}$ for $r$ dividing $q+1$,
$\S_{4}$ for $(r,q)=(3,7)$, and $\A_{5}$ for $(r,q)=(5,11)$. Then by ~\cite{a:ABD-PSL2,a:delan-linear-space,a:delan-sax,a:delan-flagt-simple} and the same argument as in~\cite[Lemmas 3.1,3.3 and 3.7]{a:Non-symmetric}, the assertion holds.
\end{proof}

\begin{table}%[h]
%\centering
\scriptsize
\caption{Possible groups $G$ and $H$ in Proposition~\ref{prop:psl2}}\label{tbl:psl2}
\begin{tabular}{llc}
\hline
$G$& $H$ &$|G:H|$ \\
\hline
$\PGL(2,7)$ & $\N_{G}(\D_{6})=\D_{12}$ & $28$\\
$\PGL(2,7)$& $\N_{G}(\D_{8})= \D_{16}$ & $21$\\
$\PGL(2,9)$ & $\N_{G}(\D_{10})= \D_{20}$ & $36$\\
$\PGL(2,9)$ & $\N_{G}(\D_{8})= \D_{16}$ & $45$\\
$\M_{10}$ & $\N_{G}(\D_{10})= \C_{5}\rtimes  \C_{4}$ & $36$\\
$\M_{10}$ & $\N_{G}(\D_{8})= \C_{8} \rtimes  \C_{2}$ & $45$\\
$\PGaL(2,9)$ & $\N_{G}(\D_{10})= \C_{10} \rtimes  \C_{4}$ & $36$\\
$\PGaL(2,9)$ & $\N_{G}(\D_{8})= \C_{8} \cdot \Aut(\C_{8})$ & $45$\\
$\PGL(2,11)$ & $\N_{G}(\D_{10})= \D_{20}$ & $66$\\
$\PGL(2,11)$& $\N_{G}(\A_{4})= \S_{4}$ & $55$\\
\hline
\end{tabular}
\end{table}
Note by \eqref{eq:iso} and Propositions~\ref{prop:alt-spor} and~\ref{prop:psl2} that for $X=\PSL_{n}(q)$ in Proposition~\ref{prop:psl} below, we only need to consider the case where $n\geq 3$ and $(n,q)\neq (3,2)$ and $(4,2)$.

\begin{proposition}\label{prop:psl}
Let $\Dmc$ be a nontrivial $2$-design with prime replication number $r$. Suppose that $G$ is an automorphism group of $\Dmc$ of almost simple type with socle $X=\PSL_{n}(q)$ with $ n\geq 3$ and $(n,q)\neq (3,2)$ and $(4,2)$. If $G$ is flag-transitive and $H=G_{\alpha}$ with $\alpha$ a point of $\Dmc$, then $H\cap X\cong \,^{\hat{}}[q^{n-1}]{:}\SL_{n-1}(q){\cdot} (q-1)$ is a parabolic subgroup, $v=(q^{n}-1)/(q-1)$ and $r$ is a primitive divisor of $(q^{n-1}-1)/(q-1)$.
\end{proposition}
\begin{proof}
Suppose that $H_{0}=H\cap X$, where $H=G_{\alpha}$ with $\alpha$ a point of $\Dmc$. Since the point-stabiliser $H$ is maximal in $G$, by Lemma \ref{lem:coprime}, Proposition~\ref{prop:class-s} and \cite[Theorem 7 and Proposition 4.7]{a:AB-Large-15}, one of the following holds:
\begin{enumerate}[\rm (1)]
\item $H \in \Cmc_{1}{\cup}\Cmc_{8}$;
\item $H$ is a $\Cmc_2$-subgroup of type $\GL_{n/t}(q)\wr\S_t$ with $t=2,3$;
\item $H$ is a $\Cmc_3$-subgroup of type $\GL_{n/t}(q^t)$ with $t=2,3$;
\item $H$ is a $\Cmc_5$-subgroup of type $\GL_{n}(q_{0})$ with $q=q_{0}^{t}$ and $t=2, 3$.
\end{enumerate}
In what follows, we analyse each of these possible cases separately. \smallskip

\noindent\textbf{(1)} Let $H$ be in $\Cmc_{1}$. In this case $H$ is reducible, that is, $H \cong P_{m}$ stabilises a subspace of $V$ of dimension $m$ with $2m\leq n$ or $G$ contains a graph automorphism and $H$ stabilises a pair $\{U,W\}$ of subspaces of dimension $m$ and $n-m$ with $2m<n$.\smallskip

Suppose first that $H\cong P_{m}$ for some $2m\leq n$. Then by~\cite[Proposition 4.1.17]{b:KL-90}, we have that
\begin{align*}
H_{0}\cong \,^{\hat{}}[q^{m(n-m)}]{:}\SL_{m}(q){\times }\SL_{n-m}(q){\cdot} (q-1).
\end{align*}
Then by \eqref{eq:v}, we observe that $v>q^{m(n-m)}$. Note that $r$ is prime and $2m\leq n$. By Lemma~\ref{lem:coprime}, we conclude that $r$ divides $|H_{0}|_{p'}$, and so does $q^{i}-1$, for some $1\leq i\leq \max\{m,n-m\}=n-m$. Therefore, $r\leq q^{n-m}-1$. It follows from Lemma~\ref{lem:six}(c) that $q^{m(n-m)}< \lambda v<r^2\leq (q^{n-m}-1)^2<q^{2(n-m)}$. Thus $q^{m(n-m)}<q^{2n-2m}$, and so $m(n-m)<2n-2m$ implying that $m=1$. In this case, $v=(q^{n}-1)/(q-1)$, and so $v-1=q(q^{n-1}-1)/(q-1)$. Since $r$ divides $v-1$, it follows from Lemma~\ref{lem:coprime} that $r$ divides $(q^{n-1}-1)/(q-1)$. If $r$ divides $q^{i}-1$, for some $i\leq n-2$, then $r$ is a divisor of $\gcd(q^{i}-1,q^{n-1}-1)=q^{\gcd(i,n-1)}-1$, and since $\gcd(i,n-1)\leq [(n-1)/2]$, we conclude that $r\leq q^{[(n-1)/2]}-1$. Since also $v<r^{2}$, it follows that $q^{n}-1<(q-1)(q^{[(n-1)/2]}-1)^{2}$ implying that $q^{n}<q^{n-1}(q-1)$, which is impossible. Therefore, $r$ is a primitive divisor of $(q^{n-1}-1)/(q-1)$.\smallskip

Suppose now that $H$ stabilises a pair $\{U,W\}$ of subspaces of dimension $m$ and $n-m$ with $2m<n$. Assume first that $U\subset W$. In this case, $G$ has a subdegree which is a power of $p$, and so $r=p$ by Lemma~\ref{lem:six}(d). But this is impossible by Lemma \ref{lem:coprime}. Assume now that $V=U\oplus W$. Then by~\cite[Proposition 4.1.4]{b:KL-90}, $H_{0}\cong \,^{\hat{}}\SL_{m}(q)\times \SL_{n-m}(q)$. Then Lemma~\ref{lem:coprime} implies that $r$ divides $|H_{0}|_{p'}$, and so does $q^{i}-1$, for some $1\leq i\leq \max\{m,n-m\}=n-m$. Therefore, $r\leq q^{n-m}-1$. Note by~\eqref{eq:v} that $v>q^{2m(n-m)}$. Then Lemma~\ref{lem:six}(c) implies that $q^{2m(n-m)}< \lambda v<r^2\leq (q^{n-m}-1)^2<q^{2(n-m)}$, that is to say, $q^{2m(n-m)}<q^{2n-2m}$, and so $2m(n-m)<2n-2m$ implying that $m<1$, which is impossible.\smallskip

Let now $H$ be in $\Cmc_{8}$. In this case $H$ is a classical group. Then by~\cite[Propositions 4.8.3, 4.8.4 and 4.8.5]{b:KL-90}, $H_{0}$ is isomorphic to one of the following groups:
\begin{enumerate}[]
\item $\,^{\hat{}}\Sp_{2m}(q){\cdot} \gcd(m, q-1)$ with $n=2m\geq 4$,
\item $\PSO_{2m+1}(q)$ with $m\geq 1$ and $q$ odd,
\item $\PSO_{2m}^{\e}(q)$ with $m\geq 2$, $q$ odd and $\e =\pm$,
\item $\,^{\hat{}}\SU_{n}(q^{\frac{1}{2}}){\cdot} \gcd(n, q^{\frac{1}{2}}-1)$ with $n\geq 3$ and $q$ a square.
\end{enumerate}
Define
\[
g(q)=
\begin{cases}
q^{2m^{2}-m-3},& \text{if} \quad H_{0}= \,^{\hat{}}\Sp_{2m}(q){\cdot} \gcd(m, q-1), \\
q^{2m^{2}+3m-1},& \text{if} \quad H_{0}= \PSO_{2m+1}(q), \\
q^{2m^{2}+m-2},& \text{if} \quad H_{0}= \PSO_{2m}^{\e}(q), \\
q^{\frac{n^{2}-4}{2}},& \text{if} \quad H_{0}= \,^{\hat{}}\SU_{n}(q^{\frac{1}{2}}){\cdot} \gcd(n, q^{\frac{1}{2}}-1).
\end{cases}
\]
We now apply \cite[Corollaries 4.2-4.3]{a:AB-Large-15} and conclude that $v>g(q)$ in relevant cases for $H_{0}$. Moreover, by Lemma~\ref{lem:coprime}, we conclude that $r\leq q^{s}+1$, where $s=n$ in the unitary case and $s=m$ in all other cases. Since $v<r^{2}$, it follows that $g(q)<(q^{s}+1)^{2}$, and hence we obtain the following possibilities:
\begin{enumerate}[(i)]
\item $\,^{\hat{}}\Sp_{4}(q){\cdot} \gcd(2, q-1)$,
\item $\,^{\hat{}}\SU_{3}(q^{\frac{1}{2}}){\cdot} \gcd(3, q^{\frac{1}{2}}-1)$.
\end{enumerate}
If (i) holds, then $v=q^2(q^3-1)/\gcd(2, q-1)$ and $r\leq q^{2}+1$, and so Lemma~\ref{lem:six}(c) yields $q^2(q^3-1)/\gcd(2, q-1)\leq \lambda v<r^2\leq (q^{2}+1)^2$, that is to say, $q^2(q^3-1)<\gcd(2, q-1)(q^{2}+1)^2$, which is impossible. If (ii) holds, then  $v\geq q^{\frac{3}{2}}(q+1)(q^{\frac{3}{2}}-1)/3$ with $q$ a square. Since $r\leq q^{\frac{3}{2}}+1$, it follows that $v>r^{2}$, which is a contradiction.\smallskip

\noindent\textbf{(2)} Let $H$ be  a $\Cmc_2$-subgroup of type $\GL_{n/t}(q)\wr\S_t$ with $t=2,3$. Then by~\cite[Proposition 4.2.9]{b:KL-90}, we have that $H_{0}\cong \,^{\hat{}}\SL_{m}(q)^t{\cdot} (q-1)^{t-1}{\cdot} \S_{t}$ with $n=mt$.\smallskip

It follows from~\eqref{eq:v} that $v>q^{n(n-m)}/(t!)$. Note here that $r$ is prime and $2m<n$. Then by Lemma~\ref{lem:coprime}, we conclude that $r\leq t(q^m-1)$, and so Lemma~\ref{lem:six}(c) implies that $q^{n(n-m)}/(t!)\leq\lambda v<r^2\leq t^2(q^m-1)^2<t^2q^{2m}$. Thus $q^{m^2t(t-1)-2m}<t^3(t-1)!$, where $t=2,3$. This inequality holds only for $(m, t, q)=(1, 3, 2)$ in which case we have that $r\leq 3$ and $v=28$, and so $v=28>9\geq r^{2}$, which is a contradiction.

\noindent\textbf{(3)} Let $H$ be  a $\Cmc_3$-subgroup of type $\GL_{n/t}(q^t)$ with $t=2,3$. Then by~\cite[Proposition 4.3.6]{b:KL-90}, we have that $H_{0}\cong \,^{\hat{}}\SL_{m}(q^t){\cdot} (q^t-1)(q-1)^{-1}{\cdot} t$ with $t$ prime and $n=mt$.\smallskip

By~\eqref{eq:v}, we have that $v>q^{n(n-m)/2}$. Since $r$ is prime, Lemma~\ref{lem:coprime} implies that $r\leq q^{mt}-1$, and so  Lemma~\ref{lem:six}(c) yields  $q^{n(n-m)/2}\leq \lambda v<r^2\leq (q^{mt}-1)^2<q^{2mt}$. This forces $n(n-m)<4n$, and so $m(t-1)=n-m<4$. This inequality holds only for $(m, t)\in\{(1, 2), (1, 3), (2, 2), (3, 2)\}$ and considering the fact that $n\geq 3$, we have $(m, t) \in \{(1, 3), (2, 2), (3, 2)\}$. In these cases, the groups $X$ and $H_{0}$ are as in Table~\ref{tbl:psl-c3}, and in each case we observe that the condition $v<r^2$ does not hold except for $X =\PSL_{3}(q)$ and $H\cap X\cong\, ^{\hat{}}(q^2+q+1){:}3$. In which case $v=q^3(q^2-1)(q-1)/3$, and $r$ divides $q^2+q+1$. It follows from Lemma \ref{lem:six}(c) that $q^3(q^2-1)(q-1)/3\leq \lambda v<r^2\leq (q^2+q+1)^2$, and so $q^3(q^2-1)(q-1)<3(q^2+q+1)^2$. This inequality holds only for $q=2, 3$. Since $(n,q)\neq (3,2)$, we obtain $q=3$ in which case $X=\PSL_{3}(3)$, $H\cap X\cong 13{:}3$ and $v=144$. Since $r$ divides both $v-1=143$ and $|H\cap X|_{3'}$ which is a divisor of $13$, it follows that $r=13$. We now apply Lemma~\ref{lem:six}(a)-(b) and conclude that $b=156$, $k=12$ and $\lambda=1$ but none of $\PSL_{3}(3)$ and $\PSL_{3}(3){:}2$ has subgroups of index $156$, which is a contradiction. \smallskip
%If $q=2$, then $v=8$, and since $r$ divides both $v-1=7$ and $|H\cap X|_{2'}$ which is a divisor of $3{{\cdot}}7$, it follows that $r=7$. We now apply Lemma~\ref{lem:six}(a)-(b) and conclude that $b=14$, $k=4$ and $\lambda=3$. It follows from \cite{b:Handbook, a:Non-symmetric} that $\Dmc$ is a unique $2$-design with parameters $(8, 14, 7, 4, 3)$ and flag-transitive and point-primitive automorphism group $G=\PSL(3,2){\cong}\PSL(2,7)$.

\begin{table}
\scriptsize
\caption{Some large maximal $\Cmc_{3}$-subgroups of some linear groups.}\label{tbl:psl-c3}
\begin{tabular}{lllll}
\hline
Line &
$X$ &
$H\cap X$ &
$v$ &
$u_{r}$ \\
\hline
$1$ &
$\PSL_{3}(q)$ &
$^{\hat{}}(q^2+q+1){:}3$ &
$q^3(q^2-1)(q-1)/3$ &
$q^2+q+1$
\\
$2$ &
$\PSL_{4}(q)$ &
$^{\hat{}}\SL_{2}(q^2){{\cdot}}(q+1){\cdot} 2$ &
$q^4(q^3-1)(q-1)/2$ &
$q^{2}+1$
\\
$3$ &
$\PSL_{6}(q)$ &
$^{\hat{}}\SL_{3}(q^2){{\cdot}}(q+1){\cdot} 2$&
$q^6(q^5-1)(q^3-1)(q-1)$&
$q^{5}-1$
\\
\hline
%\scriptsize
%Note: & \multicolumn{10}{p{13cm}}{\scriptsize The last column addresses to references in which a design with the parameters in the line has been constructed.}
\end{tabular}
\end{table}

\noindent\textbf{(4)} Let $H$ be  a $\Cmc_5$-subgroup of type $\GL_{n}(q_{0})$ with $q=q_{0}^{t}$ and $t=2, 3$. Then by~\cite[Proposition 4.5.3]{b:KL-90}, we see that $H_{0}\cong\,^{\hat{}} \SL_{n}(q_{0}){\cdot} \gcd (n, q-1/(q_{0}-1))$ with $q=q_{0}^{t}$ and $t=2, 3$.\smallskip

By \cite[Corollary 4.3]{a:AB-Large-15}, we have that $v>q_{0}^{m(n^{2}-2)-n^{2}+1}$, and since $r\leq q_{0}^{n}-1$, the inequality  $v<r^{2}$ forces that $n=3$ and $t=2$. In this case, $v=q_{0}^{3}(q_{0}^{3}+1)(q_{0}^{2}+1)/\gcd(3,q_{0}+1)$, and this case can also be ruled out as $v>q_{0}^{6}>r^{2}$.
\end{proof}

\begin{proposition}\label{prop:unitary}
Let $\Dmc$ be a nontrivial $2$-design with prime replication number $r$. Suppose that $G$ is an automorphism group of $\Dmc$ of almost simple type with socle $X$. If $G$ is flag-transitive, then the socle $X$ cannot be $\PSU_{n}(q)$ with $n\geq 3$ and $(n, q)\neq (3, 2), (4, 2)$.
\end{proposition}
\begin{proof}
Let $H_{0}=H\cap X$, where $H=G_{\alpha}$ for some point $\alpha$ of $\Dmc$. Then by Lemma \ref{lem:coprime}, Proposition~\ref{prop:class-s}, and \cite[Theorem 2.7 and Proposition 4.17]{a:AB-Large-15}, one of the following holds:
\begin{enumerate}[\rm (1)]
\item $H \in \Cmc_{1}$;
\item $H$ is a $\Cmc_2$-subgroup of type $\GU_{n/t}(q)\wr\S_t$, where $t=2$, or $t=3$ and $q\in\{2,3,4,5,7,9,13,16\}$, or $4\leq n=t\leq 11$ and $q\in\{2,3,4,5\}$;
\item $H$ is a $\Cmc_2$-subgroup of type $\GL_{n/2}(q^2)$;
\item $H$ is a $\Cmc_3$-subgroup of type $\GU_{n/3}(3^3)$ with $n$ odd;
\item $H$ is a $\Cmc_5$-subgroup of type $\GU_{n}(q_{0})$ with $q=q_{0}^3$;
\item $H$ is a $\Cmc_5$-subgroup of type $\Sp_n(q)$ or $\O_{n}^{\e}(q)$ with $\e {\in} \{\circ, -, +\}$.
\end{enumerate}
We analyse each of these possible cases separately and arrive at a contradiction in each case. \smallskip

\noindent\textbf{(1)} Let $H$ be in $\Cmc_{1}$. Then $H$ is reducible and it is either a parabolic subgroup $P_{m}$, or the stabiliser $N_{m}$ of a nonsingular subspace.\smallskip

Assume first that $H_{0}\cong P_{m}$, for some $2m\leq n$. Then since $n\geq 3$, by Lemma~\ref{lem:min-deg}, we have that $v>q^2+q+1$. On the other hand, Lemma~\ref{lem:subdeg} says that there is a subdegree which is a power of $p$. Since $r$ is prime, Lemma~\ref{lem:six}(d) implies that $r=p$, and this is impossible by Lemma~\ref{lem:coprime}.
% It follows from Lemma~\ref{lem:six}(c) that $q^2+q+1<\lambda v< r^2=p^2$, and so $q^{2}+q+1<p^2$, which is a contradiction.\smallskip

Assume now that $H_{0}\cong N_{m}$ with $2m <n$. Then by~\cite[Proposition 4.1.4]{b:KL-90}, we have that
\begin{align*}
H_{0}\cong \,^{\hat{}}\SU_{m}(q){\times}\SU_{n-m}(q){{\cdot}}(q+1).
\end{align*}
Here by~\eqref{eq:v}, we have that $v>q^{m(n-m)}$. Note here that $r$ is prime and $2m<n$. Thus by Lemma~\ref{lem:coprime}, we conclude that $r$ divides $|H_{0}|_{p'}$, and so does $q^{i}-(-1)^{i}$, for some $1\leq i\leq \max\{m,n-m\}=n-m$. Therefore, $r\leq q^{n-m-1}+1$. It follows from Lemma~\ref{lem:six}(c) that $q^{m(n-m)}\leq \lambda v< r^2\leq (q^{n-m-1}+1)^{2}<q^{2n-2m}$. Thus $q^{m(n-m)}<q^{2n-2m}$, and so $m<2$. Thus $m=1$ in which case $v=q^{n-1}(q^{n}{-}{(-1)}^{n})/(q+1)$ and $r\leq q^{n-2}+1$. Therefore Lemma~\ref{lem:six}(c) implies that $q^{n-1}(q^{n}-(-1)^n)/(q+1)\leq \lambda v< r^2\leq (q^{n-2}+1)^2$. Thus $q^{n-1}(q^{n}-(-1)^n)<(q+1)(q^{n-2}+1)^2$ and since $(q+1)(q^{n-2}+1)^2<27q^{2n-3}/8$, we conclude that $q^{n}-(-1)^n<27q^{n-2}/8$, which is a contradiction. \smallskip

\noindent\textbf{(2)} Let $H$ be  a $\Cmc_2$-subgroup of type $\GU_{n/t}(q)\wr\S_t$. Then by~\cite[Proposition 4.2.9]{b:KL-90}, we have that
\begin{align*}
H_{0}\cong\,^{\hat{}}\SU_{n/t}(q)^{t}{{\cdot}}{(q+1)}^{t-1}{{\cdot}}\S_{t}.
\end{align*}
It follows from~\eqref{eq:v} that $v>q^{n^2(t-1)/2t}/(t!)$. Note by Lemma~\ref{lem:coprime} that $r$ divides $t$ or $q^{n/t}-(-1)^{n/t}$.

If $t=2$, then $r$ must divide $q^{n/2}-(-1)^{n/2}$, and so $r\leq q^{n/2}+1$. So Lemma~\ref{lem:six}(c) implies that $q^{n^2/4}<2v<2r^2\leq 2(q^{n/2}+1)^{2}<2q^{n+2}$, then $q^{n^{2}}<2^{4}\cdot q^{4n+8}$, and this is true only for $n=4$. In this case, $v=q^4(q^2+1)(q^2-q+1)/2$ and $r\leq q+1$, and hence $r^{2}<v$, which is contradiction.

If $t=3$, then either $r=3$, or $r$ is a divisor of $q^{n/3}-(-1)^{n/3}$. In the former case, as $v<r^{2}$, we conclude that $q^{n^{2}}<54^{3}$, and since $n$ is a multiple of $3$, we have that $(n,q)=(3,3)$. Then $v=63$, and so $v>r^{2}=9$, which is a contradiction. Assume now that $r$ divides $q^{n/3}-(-1)^{n/3}$. Since $v<r^{2}$, it follows that $q^{n^{2}}<6^{3}(q^{n/3}+1)^{6}$, and this inequality holds only when $n=3$ and $q\in\{3,4,5,7\}$. In these cases, by inspecting the index of maximal subgroups of $X=\PSU_{3}(q)$ for $q\in \{3,4,5,7\}$, we observe that $v<r^{2}$ does not hold.

If $4\leq n=t\leq 11$ and $q\in\{2,3,4,5\}$, then $r$ divides $n!$ or $q+1$. Therefore, as $v<r^{2}$, we conclude that $q^{n(n-1)}<(n!)^{2}\cdot n^{4}$ or $q^{n(n-1)}<(n!)^{2}\cdot (q+1)^{4} $. These inequalities hold, only when $(n,q)=(4,2)$ and $(5,2)$ for which $v=40$ and $v=1408$, respectively. But here $r=3$ or $5$, respectively, and hence $v>r^{2}$, which is a contradiction.\smallskip

\noindent\textbf{(3)} Let $H$ be  a $\Cmc_2$-subgroup of type $\GL_{n/2}(q^2)$. Then by~\cite[Proposition 4.2.4]{b:KL-90}, we have that $H_{0}{\cong}\,^{\hat{}}(q-1)\cdot \SL_{n/2}(q^2){{\cdot}}2$ with $n$ even. Here by~\eqref{eq:v}, we observe that $v>q^{n^2/4}$. As $r$ divides $|H_{0}|_{p'}$ by Lemma~\ref{lem:coprime}, we have that $r\leq q^{n/2}+1$. Then Lemma~\ref{lem:six}(d) implies that $q^{n^2/4}<\lambda v<r^2\leq (q^{n/2}+1)^2<q^{n+2}$, and so $n^2<4n+8$. This inequality holds only for $n=4$ in which case $v=q^4(q^3+1)(q+1)/2$ and $r\leq q^2+1$,  which implies that $v>r^2$, a contradiction.\smallskip

\noindent\textbf{(4)} Let $H$ be  a $\Cmc_3$-subgroup of type $\GU_{n/3}(3^3)$ with $n$ odd. Then by~\cite[Proposition 4.3.6]{b:KL-90}, we have that $H_{0}{\cong}\,^{\hat{}}\SU_{n/3}(3^3){{\cdot}}21$, but then the inequality $|X|<|H_{0}|{{\cdot}}|H_{0}|_{3'}^2$ does not hold, which is a contradiction.\smallskip

\noindent\textbf{(5)} Let $H$ be  a $\Cmc_5$-subgroup of type $\GU_{n}(q_{0})$ with $q=q_{0}^3$. We observe by~\cite[Propositions 4.5.3]{b:KL-90}, \cite[Lemma 4.2 and Corollary 4.3]{a:AB-Large-15} and \eqref{eq:v} that $v>q_{0}^{2n^2-9}$. We moreover have by Lemma~\ref{lem:coprime} that $r$ divides $|H_{0}|_{p'}$, and so does $q_{0}^{i}-(-1)^{i}$, for some $1\leq  i\leq  n$. Therefore, $r\leq q_{0}^{n}+1$, and so Lemma~\ref{lem:six}(c) implies that $q_{0}^{2n^2-9}\leq\lambda v< r^2\leq(q_{0}^{n}+1)^2<q_{0}^{2n+2}$. Thus $q_{0}^{2n^2-9}<q_{0}^{2n+2}$, and so $2n^2<2n+11$, which is impossible.\smallskip

\noindent\textbf{(6)} Let $H$ be  a $\Cmc_5$-subgroup of type $\Sp_n(q)$ or $\O_{n}^{\e}(q)$. Then by~\cite[Propositions 4.5.5 and 4.5.6]{b:KL-90}, $H_{0}$ isomorphic to one of the following:
\begin{enumerate}[{ \quad (i)}]
\item $H_{0}{\cong}\PSO_{n}^{\circ}(q)$;
\item $H_{0}{\cong}\PSO_{n}^{\e}(q){{\cdot}}2$ with $q$ odd and $\e=\pm$;
\item $H_{0}{\cong}\,^{\hat{}}\Sp_{n}(q){{\cdot}}\gcd(\frac{n}{2}, q+1)$ with  $n$ even.
\end{enumerate}
Assume first that $H_{0}{\cong}\PSO_{n}^{\circ}(q)$ with $n$ odd. Then \cite[Lemma 4.2 and Corollary 4.3]{a:AB-Large-15} and \eqref{eq:v} imply that $v>q^{(n^{2}+n-6)/2}$. Note by Lemma~\ref{lem:coprime} that $r\leq q^{(n-1)/2}+1$, and so Lemma~\ref{lem:six}(c) implies that $q^{(n^{2}+n-6)/2}<(q^{(n-1)/2}+1)^{2}$. This inequality does not hold for $n=3$, and if $n\geq 5$, it follows that $q^{(n^{2}+n-6)/2}<q^{n+1}$, and so $n^{2}<n+8$, which is impossible.\smallskip

Assume now that $H_{0}{\cong}\PSO_{n}^{\e}(q){{\cdot}}2$ with $q$ odd and $\e=\pm$.  Then \cite[Lemma 4.2 and Corollary 4.3]{a:AB-Large-15} and \eqref{eq:v} imply that $v>q^{(n^{2}+n-6)/2}$ and by Lemma~\ref{lem:coprime}, we have that $r\leq q^{n/2}+1$. Again, applying Lemma~\ref{lem:six}(c), we conclude that $n^{2}<n+10$, which is impossible.

Assume finally that $H_{0}{\cong}\,^{\hat{}}\Sp_{n}(q){{\cdot}}\gcd(\frac{n}{2}, q+1)$ with $n$ even. By the same argument as above, $v> q^{(n^{2}-n-6)/2}$ and $r\leq q^{n/2}+1$, and so $n^{2}<3n+10$, which is true for $n=4$. In this case, $v= q^2(q^3+1)/\gcd(2, q+1)$ and $r\leq q^2+1$. By Lemma ~\ref{lem:six}(c), we have that $q^2(q^3+1)= \gcd(2, q+1)\cdot v<\gcd(2, q+1)\cdot r^2\leq \gcd(2, q+1)\cdot(q+1)^2$ implying that $q^2(q^3+1)<\gcd(2, q+1)\cdot(q^2+1)^2$. If $q$ is even, then $q^2(q^3+1)<(q^2+1)^2$, and if $q$ is odd, then $q^2(q^3+1)<2(q^2+1)^2$. But both inequalities do not hold.
\end{proof}

\begin{proposition}\label{prop:psp}
Let $\Dmc$ be a nontrivial $2$-design with prime replication number $r$. Suppose that $G$ is an automorphism group of $\Dmc$ of almost simple type with socle $X$. If $G$ is flag-transitive, then the socle $X$ cannot $\PSp_{n}(q)'$ with $n\geq 4$ even.
\end{proposition}
\begin{proof}
Let $H_{0}=H\cap X$, where $H=G_{\alpha}$ for some point $\alpha$ of $\Dmc$. If $q=2$, then $X=\PSp_{4}(2)'\cong \A_{6}$ and this case has been treated in Proposition~\ref{prop:alt-spor}. If $q=4$, than $X=\PSp_{4}(4)'$, and so the possibilities for $H_{0}$ can be read off from \cite[p. 44]{b:Atlas}, and so it is easily to check the inequality $v<r^{2}$, and conclude that the only possible case is $H_{0}\cong \PSL_{2}(16){:}2$  for $r=17$. Since $k$ is a divisor of $vr=17\cdot 120$ and $k\leq r$, we see that $k\in\{2, 3, 4, 5, 6, 8, 10, 12, 15, 17\}$. Since also $\lambda(v-1)=r(k-1)$, we have that $(v,b,r,k,\lambda)$ is either $(120, 255, 17, 8, 1 )$, or $(120, 136, 17, 15, 2)$. The former case does not occur due to \cite{a:Saxl2002}. The latter case can be ruled out by GAP \cite{GAP4}. Therefore, in what follows, we can assume that $(n,q)\neq (4,2)$ and $(4,4)$. By Lemma \ref{lem:coprime}, Proposition~\ref{prop:class-s}, and \cite[Theorem 2.7 and Proposition 4.22]{a:AB-Large-15}, one of the following holds:
\begin{enumerate}[\rm (1)]
\item $H \in \Cmc_{1}\cup\Cmc_{8}$;
\item $H$ is a $\Cmc_2$-subgroup of type $\Sp_{n/t}(q)\wr \S_t$ with $t\in \{2, 3, 4, 5\}$;
\item $H$ is a $\Cmc_2$-subgroup of type $\GL_{n/2}(q)$;
\item $H$ is a $\Cmc_3$-subgroup of type $\Sp_{n/t}(q^t)$ with $t=2, 3$ or $\GU_{n/2}(q)$;
\item $H$ is a $\Cmc_5$-subgroup of type $\Sp_{n}(q_{0})$ with $q=q_{0}^2$.
\end{enumerate}
We now analyse each of these possible cases separately. \smallskip

\noindent\textbf{(1)} Let $H$ be in $\Cmc_{1}$. Then $H$ is parabolic or stabilises a nonsingular subspace of $V$.\smallskip

Assume first that $H{\cong}P_{m}$, the stabiliser of a totally singular $m$-subspace of $V$, with $2m\leq n$ and $m$ even. %Then by~\cite[Proposition 4.1.19]{b:KL-90}, we have that
%\begin{align*}
%H_{0}{\cong}\,^{\hat{}}[q^{m/2-3m^2/2+mn}]{:}\GL_{m}(q){\times}\Sp_{n-2m}(q)
%\end{align*}
Since $n\geq 4$, by Lemma~\ref{lem:min-deg}, we have that $v>q^2+q+1$. By Lemma~\ref{lem:subdeg}, there is a subdegree which is a power of $p$, and since $r$ is prime, Lemma~\ref{lem:six}(d) implies that $r=p$, which is a contradiction by Lemma~\ref{lem:coprime}.
%We now apply Lemma~\ref{lem:six}(c) and conclude that $q^2+q+1<v<r^2=p^2$, that is to say, $q^{2}+q+1<p^2$, which is a contradiction.\smallskip

Assume now that $H_{0}{\cong}N_{m}$, the stabiliser of a nonsingular $m$-subspace $U$ of $\V$ with $m<n$ and $m$ even. Then by~\cite[Proposition 4.1.3]{b:KL-90}, we have that
\begin{align*}
H_{0}{\cong}\,^{\hat{}}\Sp_{m}(q){\times}\Sp_{n-m}(q).
\end{align*}

It follows from~\eqref{eq:v} that $v>q^{m(n-m)}$. By Lemma~\ref{lem:coprime}, $r$ divides $|H_{0}|_{p'}$, and so does $q^{2i}-1$, for some $1\leq  i\leq  \max\{\frac{m}{2}, \frac{n-m}{2}\}{=}\frac{n-m}{2}$.  Therefore, $r\leq q^{(n-m)/2}+1$. We now apply Lemma~\ref{lem:six}(c) and conclude that $q^{m(n-m)}\leq \lambda v<r^2\leq (q^{(n-m)/2}+1)^2<q^{n-m+2}$. Thus $m(n-m)<n-m+2$, which is impossible.\smallskip

Let now $H$ be in $\Cmc_{8}$. Then by~\cite[Proposition 4.8.6]{b:KL-90}, we have that $H_{0}\cong\O_{n}^{\e}(q)$ with $q$ even. In this case,  $v=q^{m}(q^{m}+\e)/2$, where $n=2m$. It follows from proof of \cite[Proposition 1]{a:LPS2} that $X$ has the subdegrees $(q^m-\e1)(q^{m-1}+\e1)$ and $q^{m-1}(q^{m}-\e1)$. By Lemma~\ref{lem:six}(d), the parameter $r$ divides $c(q^m-\e1)$, where $c=\gcd(q-2, q^{m-1}+\e1 )$, and hence Lemma \ref{lem:divisible} implies that $r$ is divisible by the index of a parabolic subgroup in $\O_{n}^{\e}(q)$ and this is clearly not possible.\smallskip

\noindent\textbf{(2)} Let $H$ be  a $\Cmc_2$-subgroup of type $\Sp_{n/t}(q)\wr\S_t$ with $t\in \{2, 3, 4, 5\}$. Then by~\cite[Proposition 4.2.10]{b:KL-90}, we have that $H_{0}\cong \,^{\hat{}}\Sp_{n/t}(q)\wr \S_t$.

It follows from~\eqref{eq:v} that $v>q^{m^2t(t-1)/2}/(t!)$ with $n=mt$ and $m$ even. Lemma~\ref{lem:coprime} implies that $r\leq t(q^{m/2}+1)$, and so by Lemma~\ref{lem:six}(d), we conclude that $q^{m^2t(t-1)/2}\leq (t!)\cdot v< (t!)\cdot r^2\leq t^2\cdot (t!)\cdot (q^{m/2}+1)^2$. Therefore $q^{m^2t(t-1)/2}< t^2\cdot (t!)\cdot (q^{m/2}+1)^2$. If $(m,t)=(2,2)$, then we have that $q=2$ or $3$, and so $X$ is isomorphic to $\PSp_{4}(2)'\cong \A_{6}$ or $\PSp_{4}(3)\cong \PSU_{4}(2)$, which are not the case by Propositions~\ref{prop:alt-spor} and~\ref{prop:unitary}. If $(m,t)\neq (2,2)$, then $q^{m^2t(t-1)/2}< t^2\cdot (t!)\cdot (q^{m/2}+1)^2<t^2\cdot (t!)\cdot q^{m+2}$, and hence $q^{m^2t(t-1)-2m-4}<t^4\cdot (t!)^{2}$, where $t\in\{2,3,4,5\}$. This inequality does not hold if $(m, t)\neq (2,2)$, which is a contradiction.\smallskip

\noindent\textbf{(3)} Let $H$ be  a $\Cmc_2$-subgroup of type $\GL_{n/2}(q)$. Then by~\cite[Proposition 4.2.5]{b:KL-90}, we have that $H_{0}\cong\,^{\hat{}}\GL_{m}(q){\cdot} 2$ with $n=2m$.

Here by \ref{eq:v}, we have that $v>q^{m(m+1)}/2$, and by Lemma~\ref{lem:coprime}, we conclude that $r\leq q^{m}-1$, and so Lemma~\ref{lem:six}(d) implies that $q^{m(m+1)}\leq 2v< 2r^2\leq 2(q^{m}-1)^2<2q^{2m}$. Thus $m^{2}<m+1$, which is impossible.\smallskip

\noindent\textbf{(4)} Let $H$ be  a $\Cmc_3$-subgroup of type $\Sp_{n/2}(q^2)$, $\Sp_{n/3}(q^3)$ or $\GU_{n/2}(q)$. Then by~\cite[Proposition 4.3.7 and 4.3.10]{b:KL-90}, $H_{0}$ isomorphic to one of the following subgroups:
\begin{enumerate}[(i)]
\item $^{\hat{}}\GU_{m}(q){\cdot} 2$ with $m=n/2$ and $q$ odd;
\item $\PSp_{m}(q^t){\cdot} t$ with $m=n/t$ even and $t=2, 3$.
\end{enumerate}

Assume first that $H_{0}\cong \,^{\hat{}}\GU_{m}(q){\cdot} 2$ with $m=n/2$ and $q$ odd. Note that $r$ is odd prime. Then by Lemma \ref{lem:coprime}, we note that $r$ is coprime to $p$, and so by Lemma~\ref{lem:divisible}, the stabiliser of a block under $H_{0}$ is contained in a parabolic subgroup of $\GU_{m}(q)$, and since the indices of parabolic subgroups of unitary groups of odd characteristic are even, it follows that $r$ is even but here $v-1$ is odd, which is a contradiction.\smallskip

Assume now that $H_{0}\cong\PSp_{m}(q^t){\cdot} t$ with $m=n/t$ even and $t=2, 3$. By \cite[Lemma 4.2 and Corollary 4.3]{a:AB-Large-15} and \eqref{eq:v}, we have that $v>q^{m^{2}t(t-1)/2}/4t$, where $t=2, 3$. Note by Lemma~\ref{lem:six}(d) that $r\neq 2,3$ unless $(m,t,q)=(2,2,2)$ in which case $X=\PSp_{4}(2)'\cong \A_{6}$ and this case has been treated in Proposition~\ref{prop:alt-spor}. Then Lemma~\ref{lem:coprime} implies that $r\leq q^{mt/2}+1$. We again apply Lemma~\ref{lem:six}(d) and conclude that $q^{m^{2}t(t-1)/2}\leq 4t v< 4tr^2\leq 4t(q^{mt/2}+1)^2<4tq^{mt+2}$, and so $q^{m^2t(t-1)}<16t^{2}q^{2mt+4}$, where $t=2,3$. This inequality holds only for $(m, t)=(2, 2)$. In this case, $X=\PSp_{4}(q)'$ and $H_{0}=\PSp_{2}(q^{2}){\cdot}2$ with $v=q^2(q^2-1)/2$ and $q\neq 2,4$. According to \cite[p. 329]{a:Saxl2002}, the nontrivial subdegrees are
\begin{align*}
(q^2+1)(q-1), q(q^2+1)/2 \text{ and }q(q^2+1)& \text{ if } q \text{ is odd, and} \\
(q^2+1)(q-1) \text{ and }  q(q^2+1) \text{ if } q \text{ is even.}
\end{align*}
Since $r$  is a prime divisor of $q^2+1$ and $r^2 > v=q^2(q^2-1)/2$, it follows that $2(q^2+1)^{2}>s^{2}q^2(q^2-1)$, where $rs=q^{2}+1$, for some positive integer $s$, but this is true only when $s=1$. Hence $r=q^{2}+1$ is a Fermat prime and $q=2^{a}$. We now apply Lemma \ref{lem:six}(a) and conclude that $k=1+\lambda(q^2-2)/2$, and since $k\leq r$, we must have $\lambda(q^2-2)\leq 2q^2$. Excluding the case where $X=\PSp_{4}(2)'\cong \A_{6}$, this inequality implies that $\lambda=1,2$. By \cite{a:Saxl2002}, we can assume that $\lambda= 2$. In this case by Lemma \ref{lem:six}(a)-(b), we have that $k=q^2-1$ and $b=q^2(q^2+1)/2$.

Let $K=\PSp_{2}(q^{2})\cong \PSL_{2}(q^{2})$. Then by Lemma \ref{lem:divisible}, $K_{B}$ is contained in a parabolic subgroup $P$ of index $q^{2}+1$ in $K$ and $|K:P|=r=q^{2}+1$. Therefore, $P$ is a subgroup of $H$ of index $2(q^{2}+1)$. Since $|H:H_{B}|=r=q^{2}+1$ is prime and $|H:K|=2$, $K$ does not contain $H_{B}$. Therefore, $H=KH_{B}$, and so $q^{2}+1=|H:H_{B}|=|K:K_{B}|=|K:P|\cdot |P:K_{B}|=(q^{2}+1)\cdot |P:K_{B}|$. This implies that $K_{B}=P\cong q^{2}:(q^{2}-1)$. Since $b=q^{2}(q^{2}+1)/2$, we conclude that $|G:G_{B}|=q^{2}(q^{2}+1)/2$, and so $|X:X_{B}|=q^{2}(q^{2}+1)/(2c)$ for some divisor $c$ of $|\Out(X)|=2a=2^{t+1}$. By inspecting maximal subgroups of $X\cong \PSp_{4}(q)$ recorded in \cite[Table 8.14]{b:BHR-Max-Low}, we observe that the only maximal subgroups of $X$ whose indices divide $|X:X_{B}|=q^{2}(q^{2}+1)/(2c)$ are $\Sp_{2}(q)^{2}:2$ and $\SO_{4}^{+}(q)$, both of index $q^{2}(q^{2}+1)/2$, and this implies that $c=1$ and $X_{B}$ is maximal in $X$ and is isomorphic to $\Sp_{2}(q)^{2}:2$ or $\SO_{4}^{+}(q)$. Let now $R$ be a Sylow $r$-subgroup of $X$, where $r=q^{2}+1$ is a Fermat prime. We now apply \cite[Theorem 1.3.4]{b:=Wielandt64} and observe that $q^{2}+1$ divides $|R:R_{B}|$, and so $R$ is semiregular, and hence $C:=C_{X}(R)$ is transitive on the block set $\Bmc$.  Therefore, $X=X_{B}C$. Suppose that $M$ is a maximal subgroup of $X$ containing $C$. Then $X=X_{B}M$, and since $X_{B}$ is maximal in $X$, the possible subgroups $X_{B}$ and $M$ can be read off \cite[Table 1]{b:LPS-Max90}, and so excluding the cases where $q=2,4$ and the fact that $X_{B}$ is $\Sp_{2}(q)^{2}:2$ or $\SO_{4}^{+}(q)$, we have that $(X_{B},M)$ is $(\SO^{+}_{4}(q),\PSp_{2}(q^{2}){\cdot}2)$, $(\PSp_{2}(q)\wr 2, \SO^{-}_{4}(q))$ or $(\SO^{+}_{4}(q),^{2}{\!}\B_{2}(q))$. The latter case can be ruled out as $^{2}{\!}\B_{2}(q)$ has no subgroups of order $q^{2}+1$.  If $M=\PSp_{2}(q^{2}){\cdot}2=H_{0}$, then since by \cite[p. 364]{b:BHR-Max-Low}, the subgroup $R$ is self-centralising in $H_{0}$, it follows that $C_{B}=C\cap X_{B}=R\cap X_{B}=R_{B}=1$, and hence $q^{2}(q^{2}-1)^{2}=|X:R|=|X:C|=|X_{B}:C_{B}|=|X_{B}|=2q^{2}(q^{2}-1)^{2}$, which is a contradiction. If $M=\SO^{-}_{4}(q)\cong \PSL_{2}(q^{2}){:}2$, then by inspecting the maximal subgroups of $\PSL_{2}(q^{2})$ and the fact that $2(q^{2}+1)$ is a divisor of the size of the normalizer $N$ of $R$ in $X$, we conclude that $N=\D_{2(q^{2}+1)}{:}2$. Note also by \cite[p. 364]{b:BHR-Max-Low} that $N_{H_{0}}(R)=(q^{2}+1):4$. Then $N=N_{H_{0}}(R)$, this again implies that $C=C_{\alpha}=R$, and so $|C_{B}|=1$, which has been already shown that is not the case. \smallskip

\noindent\textbf{(5)} Let $H$ be  a $\Cmc_5$-subgroup of type $\Sp_{n}(q_{0})$ with $q=q_{0}^2$. In this case by~\cite[Proposition 4.5.4]{b:KL-90}, we have that $H_{0}\cong \PSp_{n}(q_{0}){\cdot} c$ with $q=q_{0}^2$ and $c\leq 2$, (with $c=2$ if and only if $q$ is odd).

It follows from~\eqref{eq:v} that $v>q_{0}^{n(n+1)/2}/2$, and so Lemma~\ref{lem:coprime} implies that $r\leq q_{0}^{n/2}+1$. Thus by  Lemma~\ref{lem:six}(d), we have that $q_{0}^{n(n+1)}<4q_{0}^{2n+4}$, and so $n^2<n+6$, which is impossible.
\end{proof}

\begin{table}
\scriptsize
\caption{Some large maximal $\Cmc_{2}$-subgroups of some orthogonal groups.}\label{tbl:po-c2}
\resizebox{\textwidth}{!}{
\begin{tabular}{clllll}
\hline\noalign{\smallskip}
\multicolumn{1}{c}{Class} &
\multicolumn{1}{l}{$X$} &
\multicolumn{1}{l}{$H\cap X$} &
\multicolumn{1}{l}{$l_{v}$} &
\multicolumn{1}{c}{$u_{r}$}&
\multicolumn{1}{c}{Condition}
\\
\noalign{\smallskip}\hline\noalign{\smallskip}
$\Cmc_{2}$ &
$\POm_{n}^{+}(q)$ &
$\Omega_{n/2}^{\pm}(q)^2{\cdot} 2^{f}$ &
$q^{(n^2-24)/4}$&
$q^{(n+4)/4}$&
$f=2, 3$\\
$\Cmc_{2}$ &
$\POm_{n}^{+}(q)$ &
$\Omega_{n/2}(q)^2{\cdot}4$ &
$q^{(n^2-20)/4}$&
$q^{(n+2)/4}$&
$\frac{n}{2}q$ is odd\\
$\Cmc_{2}$ &
$\POm_{n}^{-}(q)$ &
$\Omega_{n/2}(q)^2{\cdot} 4$ &
$q^{(n^2-20)/4}$&
$q^{(n+2)/4}$&
$\frac{n}{2}q$ is odd\\
$\Cmc_{2}$ &
$\POm^{+}_{n}(q)$ &
$\GL_{n/2}(q)$  &
$q^{(n^2-2n)/4}/2$&
$q^{n/2}-1$&
\\
$\Cmc_{5}$ &
$\POm_{n}(q)$ &
$\Omega_{n}(q_{0}){{\cdot}}2$  &
$q_{0}^{n(n-1)/2}/4$&
$q_{0}^{n/2}+1$&
$n$ is odd and $q=q_{0}^{2}$\\
$\Cmc_{5}$ &
$\POm^{+}_{n}(q)$ &
$\POm_{n}^{\e'}(q_{0}){{\cdot}}2^{c}$  &
$q_{0}^{n(n-1)/2}/4$&
$q_{0}^{n/2}+1$&
$\e'=\pm$, $c\leq 2$ and $q=q_{0}^2$\\
$\Cmc_{2}$ &
$\POm_{8}^{+}(2)$ &
$\Omega_{2}^{-}(2)^2{\cdot} 2^{4}$  &
$2^{8}{\cdot}3^3{\cdot}5^2{\cdot}7$&
$3$&
\\
$\Cmc_{2}$ &
$\POm_{10}^{-}(2)$ &
$\Omega_{2}^{-}(2)^5{\cdot} 2^{5}$  &
$2^{15}{\cdot}3{\cdot}5^2{\cdot}7{\cdot}11{\cdot}17$&
$3$&
\\
$\Cmc_{2}$ &
$\POm_{12}^{-}(2)$ &
$\Omega_{4}^{-}(2)^2{\cdot} 2^{3}$  &
$2^{23}{\cdot}3^{4}{\cdot}5{\cdot}7{\cdot}11{\cdot}13{\cdot}17{\cdot}31$&
$5$&
\\
$\Cmc_{2}$ &
$\POm_{7}(3)$ &
$2^{6}{\cdot}\A_{7}$  &
$3^7{\cdot}13$&
$7$&
\\
$\Cmc_{2}$ &
$\POm_{7}(5)$ &
$2^{6}{\cdot}\A_{7}$  &
$3^2{\cdot}5^{8}{\cdot}13{\cdot}31$&
$7$ &
\\
$\Cmc_{2}$ &
$\POm_{8}^{+}(3)$ &
$2^{6}{\cdot}\A_{8}$  &
$3^{10}{\cdot}5{\cdot}13$&
$7$ &
\\
$\Cmc_{2}$ &
$\POm_{9}(3)$ &
$2^{7}{\cdot}\A_{9}$  &
$2{\cdot}3^{12}{\cdot}5{\cdot}13{\cdot}41$&
$7$ &
\\
$\Cmc_{2}$ &
$\POm_{10}^{-}(3)$ &
$2^{8}{\cdot}\A_{10}$  &
$3^{16}{\cdot}13{\cdot}41{\cdot}61$&
$7$ &
\\
$\Cmc_{2}$ &
$\POm_{11}(3)$ &
$2^{9}{\cdot}\A_{11}$  &
$2{\cdot}3^{21}{\cdot}11{\cdot}13{\cdot}41{\cdot}61$&
$11$&
\\
$\Cmc_{2}$ &
$\POm_{12}^{+}(3)$ &
$2^{10}{\cdot}\A_{12}$  &
$2^{25}{\cdot}7{\cdot}11{\cdot}13^2{\cdot}41{\cdot}61$&
$11$&
\\
$\Cmc_{2}$ &
$\POm_{13}(3)$ &
$2^{11}{\cdot}\A_{13}$  &
$2{\cdot}3^{31}{\cdot}5{\cdot}7{\cdot}11{\cdot}13{\cdot}41{\cdot}61{\cdot}73$&
$13$&
\\
\hline
\end{tabular}
}
\end{table}

\begin{proposition}\label{prop:orth}
Let $\Dmc$ be a nontrivial $2$-design with prime replication number $r$. Suppose that $G$ is an automorphism group of $\Dmc$ of almost simple type with socle $X$. If $G$ is flag-transitive, then the socle $X$ cannot be $\POm_{n}^{\e}(q)$ with $\e \in \{\circ, -, +\}$.
\end{proposition}
\begin{proof}
Let $H_{0}=H\cap X$, where $H=G_{\alpha}$ for some point $\alpha$ of $\Dmc$. Then by Lemma \ref{lem:coprime}, Proposition~\ref{prop:class-s} and \cite[Theorem 2.7 and Proposition 4.23]{a:AB-Large-15}, one of the following holds:
\begin{enumerate}[\rm (1)]
\item $H \in \Cmc_{1}$;
\item $H$ is a $\Cmc_2$-subgroup of type $\O_{n/t}^{\e'}(q)\wr\S_t$ and one of the following holds:
\begin{enumerate}[(a)]
\item $t=2$;
\item $(n, t, q, \e, \e')=(12, 3, 2, -, -), (10, 5, 2, -, -)$ or $(8, 4, 2, +, -)$;
\item $n=t$ and either $(n, q)= (7, 5)$, or $7\leq n \leq 13$ and $q=3$;
\end{enumerate}
\item $H$ is a $\Cmc_2$-subgroup of type $\GL_{n/2}(q)$ with $\e=+$;
\item $H$ is a $\Cmc_3$-subgroup of type $\O_{n/t}^{\e'}(q^2)$ or $\GU_{n/2}(q)$;
%\item $H$ is a $\Cmc_4$-subgroup of type $\Sp_{n/2}(q)\otimes \Sp_{2}(q)$ and $(n, \e) \in\{(12, +), (8, +)\}$,
\item $H$ is a $\Cmc_5$-subgroup of type $\O_{n}^{\e'}(q_{0})$ with $q=q_{0}^2$.
\end{enumerate}

If $H$ is a $\Cmc_i$-subgroup, for $i=2,5$, then   we apply \cite[Propositions 4.2.11, 4.2.14, 4.2.15, 4.5.8 and 4.5.10]{b:KL-90}, and obtain the pairs $(X, H\cap X)$ listed in Table~\ref{tbl:po-c2}. For each such $H\cap X$, by \cite[Lemma 4.2 and Corollary 4.3]{a:AB-Large-15}, \eqref{eq:v} and Lemma \ref{lem:coprime}, we obtain a lower bound $l_{v}$ of $v$ and an upper bound $u_{r}$ of $r$ as in the fourth and fifth column of Table~\ref{tbl:po-c2}, and then we easily observe that $l_{v}>u_{r}^{2}$, and this violates Lemma~\ref{lem:six}(c). This leaves the $\Cmc_{i}$-subgroups for $i=1,3$, which we analyse each case separately. \smallskip

\noindent\textbf{(1)} Let $H$ be in $\Cmc_{1}$. Then $H$ stabilises a totally singular $i$-subspace with $2i\leq n$ or a non-singular subspace.\smallskip

Assume first that $H$ stabilises a totally singular $i$-subspace. If $n$ is odd, we argue exactly as in the symplectic case. Let now $n=2m$ and suppose that $i<m$. Then $H=P_{i}$ unless $i=m-1$ and $\e=+$, in this case, $H=P_{m, m-1}$. Note by Lemma~\ref{lem:subdeg} that there is a subdegree which is a power of $p$ (except in the case where $\e=+$, $n/2$ is odd and $H=P_{m}$ or $H=P_{m-1}$). Then Lemma~\ref{lem:six}(d) implies that $r=p$ but this is impossible by Lemma~\ref{lem:coprime}.
% . It follows from Lemma~\ref{lem:six}(c) that $q^2+q+1<\lambda v<r^2=p^2$, and so $q^{2}+q+1<p^2$, which is a contradiction.

%By Lemma~\ref{lem:min-deg}, we have that $v>q^2+q+1$. On the other hand, Lemma~\ref{lem:subdeg} says that there is a subdegree which is a power of $p$ (except in the case where $\e=+$, $n/2$ is odd and $H=P_{m}$ or $H=P_{m-1}$). Since $r$ is prime, Lemma~\ref{lem:six}(d) implies that $r=p$. It follows from Lemma~\ref{lem:six}(c) that $q^2+q+1<\lambda v<r^2=p^2$, and so $q^{2}+q+1<p^2$, which is a contradiction.\smallskip

Assume now that $H=P_{m}$ when $X=\POm_{2m}^{+}(q)$. Note here that $P_{m}$ and $P_{m-1}$ are the stabilisers of totally singular $m$-spaces from the two different $X$-orbits. Then by \cite[Proposition 4.1.20]{b:KL-90} and \eqref{eq:v}, we have that $v>q^{m(m-1)/2}$. We conclude by \cite[p. 332]{a:Saxl2002} that $G$ has a subdegree of power $p$ when $m$ is even. This case again can be ruled out by Lemma~\ref{lem:coprime}.  This leaves the case where $m\geq 5$ is odd in which case by \cite[p. 332]{a:Saxl2002}, $G$ has a subdegree dividing $q^{m}-1$. Lemma \ref{lem:six}(d) implies that $r\leq q^{m}-1$, and so by Lemma \ref{lem:six}(c), $q^{m(m-1)/2}\leq v<r^2\leq (q^{m}-1)^2<q^{2m}$, then $q^{m(m-1)/2}<q^{2m}$, that is to say, $m^2<5m$, which is impossible.\smallskip

Assume finally that $H$ is the stabiliser of a non-singular $i$-subspace. If $i=1$, then by \eqref{eq:v}, we have that $v=q^{m}(q^{m}+\delta1)/2$ if $n=2m+1$ is odd and $v=q^{m-1}(q^{m}-\delta1)/\gcd(2, q-1)$ if $n=2m$ is even.  According to \cite[p.331-332]{a:Saxl2002}, $r\leq (q^{m}-\delta1)/2$ if $n$ is odd and $r\leq  (q^{m-1}+1)/\gcd(2, q-1)$ if $n$ is even, and so Lemma~\ref{lem:six}(c) implies that
$q^{m}(q^{m}+\delta1)<q^{2m}$ if  $n$ is odd and $q^{m-1}(q^{m}-\delta1)<q^{2m-2}+2q^{m-1}+1$ if $n$ is even.
%. Thus Lemma \ref{lem:six}(c) implies that
%\begin{align*}
%q^{m}(q^{m}+\delta1)\leq 2\lambda v<2r^2\leq \frac{(q^{m}-\delta1)^2}{2}<q^{2m},\text{ if } n=2m+1 \text{, and} \\
%q^{m-1}(q^{m}-\delta1)\leq \gcd(2, q-1)v<\gcd(2, q-1)r^2\leq \frac{(q^{m-1}+1)^2}{\gcd(2, q-1)}\leq &q^{2m-2}+2q^{m-1},\text{ if }n=2m.
%\end{align*}
In both cases, we conclude that $r^2<v$, which is a contradiction. Hence $i\geq 2$ and by \cite[Proposition 4.1.6]{b:KL-90} and \eqref{eq:v}, we have that $v>q^{i(n-i)}/4$. It follows from \cite[Proposition 4.1.6]{b:KL-90} and Lemma \ref{lem:coprime} that $r\leq q^{(n-i)/2}+1$. So Lemma \ref{lem:six}(c) implies that $q^{i(n-i)}<4v<4r^2\leq 4(q^{(n-i)/2})+1)^2<q^{n-i+4}$. Thus $(i-1)(n-i)<4$, which is impossible.\smallskip

\noindent\textbf{(4)} Let $H$ be a $\Cmc_3$-subgroup of type $\O_{n/2}^{\e'}(q^2)$ or $\GU_{n/2}(q)$. Then by \cite[Propositions 4.3.14-4.3.18 and 4.3.20]{b:KL-90}, one of the following holds:
\begin{enumerate}[(i)]
%\item $H_{0}\cong\POm_{n/2}^{\circ}(q^2){{\cdot}}2$ with $\e=\pm$;
\item $H\cong N_{G}(\Omega_{n/2}^{\e'}(q^2))$ with $\e'=\pm$ if $n/2$ is even and empty otherwise;
\item $H\cong N_{G}(\GU_{n/2}(q))$ with $\e=(-)^{n/2}$.
\end{enumerate}

Assume that $H\cong N_{G}(\Omega_{n/2}^{\e'}(q^2))$ with $\e'=\pm$ if $n/2$ is even and empty otherwise. If $q$ is odd, we apply the Tits Lemma \ref{lem:Tits} to $H$, to see that an index of a parabolic subgroup of $H$
divides $r$, it then follows that $r$ is even, but we see that $v$ is even, and which is a contradiction with the fact that $r$ dividing $v-1$. Hence $q$ is even and therefore also $n/2$ is even. Then \eqref{eq:v} implies that $v>q^{(n^2+8n-24)/8}/4$. On the other hand Lemma~\ref{lem:coprime} implies that $r\leq q^{n/2}+1$. So by Lemma~\ref{lem:six}(c), we conclude that $q^{(n^2+8n-24)/8}< 4v<4r^2\leq 4(q^{n/2}+1)^2<q^{n+4}$, and this yields $n^2<56$, which is a contradiction with the fact that $n\geq 8$.

Assume finally that $H\cong N_{G}(\GU_{n/2}(q))$ with $\e=(-)^{n/2}$. Then by \eqref{eq:v}, we have that $v>q^{(n^2-2n)/4}/2$. Also Lemma~\ref{lem:coprime} implies that $r\leq q^{n/2}+1$. Therefore by Lemma~\ref{lem:six}(c), we conclude that $q^{(n^2-2n)/4}< 2v<2r^2\leq 2(q^{n/2}+1)^2<q^{n+3}$, and this yields $n^2<6n+12$, which is a contradiction.
\end{proof}

\begin{table}%[h]
\centering
%\scriptsize
%\scriptsize
\caption{Large maximal non-parabolic subgroups $H$ of almost simple groups $G$ with socle $X$ a finite exceptional simple group of Lie type.}\label{tbl:large-exc}
\resizebox{\textwidth}{!}{
\begin{tabular}{lllll}
\hline\noalign{\smallskip}
$X$ &
$H\cap X$ &
$l_{v}$ &
$u_{r}$ &
Condition\\
\noalign{\smallskip}\hline\noalign{\smallskip}
${}^2\B_2(q)$ &
$13{:}4$ &
$2^4{\cdot}5{\cdot}7$ &
$13$ &
$q{=}8$
\\
&
$41{:}4$ &
$2^8{\cdot}5^2{\cdot}31$ &
$41$ &
$q{=}32$
\\
&
${}^2\B_2(q^{1/3})$ &
$q(q^2+1)$ &
$q+1$ &
\\
${}^2\G_2(q)$ &
$\Al_{1}(q)$, ${}^2\G_2(q^{1/3})$&
$q^2(q^2-q+1)$ &
$q+1$\\
${}^3\D_4(q)$ &
$(q^2+\e q+1)\Al_{2}^{\e}(q)$, $\Al_{1}(q^3)\Al_{1}(q)$, $\G_2(q)$ &
$q^{9}+1$ &
$q^3+1$ &
$\e{=}\pm$\\
&
${}^3\D_4(q^{1/2})$ &
$q^{6}(q^4-q^2+1)(q^3+1)$&
$q^4+q^2+1$&
$q$ square \\
&
$7^2{:}\SL_{2}(3)$ &
$2^9{\cdot}3^3{\cdot} 13$ &
$7$ &
$q{=}2$\\
${}^2\F_4(q)$&
${}^2\B_{2}(q)\wr2$, $ \B_{2}(q){:}2$, $ {}^2\F_4(q^{1/3})$ &
$q^8(q^6+1)$ &
$q^2+1$ \\
&
$\SU_{3}(8){:}2$, $\PGU_{3}(8){:}2$  &
$2^{26}{\cdot}5^2{\cdot}7{\cdot}13^2{\cdot}37{\cdot}109$ &
$19$ &
$q{=}8$\\
&
$\Al_{2}(3){:}2$, $\Al_{1}(25)$, $\A_{6}{\cdot}2^2$, $5^2{:}4{\cdot}\A_{4}$  &
$2^7{\cdot}5^2$ &
$13$ &
$q{=}2$\\
$\G_2(q)$ &
$\Al_{2}^{\e}(q)$ &
$q^{3}(q^{3}+\e1)/2$ &
$q^{2}-\e q+1$ \\
&
${}^2\G_2(q)$, $\Al_{1}(q)^{2}$, $\G_2(q^{1/b})$&
$q^3(q^3-1)(q+1)$ &
$q^3+1$ &
$b{=}2, 3$\\
&
$2^{3}{\cdot}\Al_{2}(2)$ &
$2^3{\cdot}3^2{\cdot}7^2$ &
$3$ &
$q{=}3$\\
&
$\Al_{1}(13)$, $\J_{2}$ &
$2^5{\cdot}13$ &
$13$ &
$q{=}4$\\
&
$\G_2(2)$, $2^{3}{\cdot}\Al_{2}(2)$ &
$5^6{\cdot}31$ &
$7$ &
$q{=}5$\\
&
$\G_2(2)$ &
$2^2{\cdot}7^5{\cdot}19{\cdot}43$ &
$7$ &
$q{=}7$\\
&
$\J_{1}$ &
$2^3{\cdot}3^2{\cdot}5{\cdot}11^5{\cdot}37$ &
$19$ &
$q{=}11$ \\
$\F_4(q)$ &
$\B_4(q)$, $\D_4(q)$, $\Al_{1}(q)\C_3(q)$, $\C_{4}(q)$, $\C_{2}(q^2)$, ${\C_{2}(q)}^2$, ${}^2\F_4(q)$ &
$q^{8}(q^8+q^4+1)$ &
$q^6+1$ &
\\
&
${}^3\D_4(q)$,  $\F_4(q^{1/b})$, $\Al_{1}(q)\G_{2}(q)$ &
$q^{12}(q^{8}-1)(q^4+1)$ &
$q^{8}+q^{4}+1$ &
$b=2, 3$\\
%
%&
%${}^3\D_4(2)$ &
%$2^{12}{\cdot}3^2{\cdot}5^2{\cdot}17$ &
%$13$ &
%$q{=}3$
%\\
%
&
$\Al_{3}(3)$, ${}^3\D_4(2)$, $\D_{4}(2)$, $\A_{9-10}$, $\Al_{3}(3)$, $\J_{2}$, $\S_{6}{\wr}\S_{2}$ &
$2^{11}{\cdot}7{\cdot}13{\cdot}17$ &
$13$ &
$q{=}2$
\\
$\E_6^{\e}(q)$ &
$\Al_{1}(q)\Al_{5}^{\e}(q)$, $\F_4(q)$, $(q-\e1)\D_5^{\e}(q)$, $\C_4(q)$ &
$q^{12}(q^9-\e1)$ &
$q^6+1$ &
$\e{=}\pm$\\
&
$(q^2+\e q+1){\cdot}{}^3\D_4(q)$ &
$q^{24}(q^5-\e1)$ &
$q^8+q^4+1$ &
$(\e,q){\neq}(-,2)$
\\
&
$(q-\e1)^2{{\cdot}}\D_4(q)$ &
$q^{24}(q^{12}-1)$ &
$q^3+1$ &
$(\e,q){\neq} (+,2)$
\\
&
$\E_6^{\e'}(q^{1/2})$&
$q^{36}(q^{12}+1)$ &
$q^9+1$&
$\e{=}+$, $\e'{=}\pm$
\\
&
$\E_6^{\e}(q^{1/3})$ &
$q^{36}(q^{18}-\e1)$&
$q^9+1$ &
$\e{=}\pm$\\
&
$\Fi_{22}$, $\B_3(3)$, $\A_{12}$, $\J_{3}$  &
$2^{16}{\cdot}3^2{\cdot}7{\cdot}13{\cdot}19$&
$13$ &
$(\e, q){=}(-,2)$
\\
$\E_7(q)$ &
$(q-\e1)\E_{6}^{\e}(q)$,  $\Al_{1}(q)\D_6(q)$, $\Al_{7}^{\e}(q)$, $\Al_{1}(q)\F_4(q)$, $\E_7(q^{1/b})$  &
$q^{27}(q^{14}-1)$ &
$q^{15}+1$ &
$\e{=}\pm$, $b{=}2,3$
\\
&
$\Fi_{22}$ &
$2^{46}{\cdot}3^2{\cdot}7^2{\cdot}19{\cdot}31{\cdot}43{\cdot}73{\cdot}127$ &
$13$ &
\\
$\E_8(q)$ &
$\Al_{1}(q)\E_7(q)$, $\D_{8}(q)$, $\Al_{2}^{\e}(q)\E_6^{\e}(q)$, $\E_8(q^{1/b})$ &
$q^{56}(q^{30}-1)$ &
$q^{15}+1$ &
$\e{=}\pm$, $b{=}2,3$
\\
\noalign{\smallskip}\hline
\end{tabular}
}
\end{table}

\subsection{Exceptional groups}\label{sec:exceptional}
In this section, we suppose that $\Dmc$ is a nontrivial $2$-design with prime replication number $r$ and $G$ is an almost simple automorphism group $G$ whose socle $X$ is a finite exceptional simple group of Lie type. Here, it is convenient to adopt the Lie notation for groups of Lie type, so for example, we will write $\Al_{n-1}^{-}(q)$ in place of $\PSU_n(q)$, $\D_{n}^{-}(q)$ instead of $\POm_{2n}^{-}(q)$, and $E_{6}^{-}(q)$ for ${}^{2}\E_{6}(q)$. Also recall that the type of a subgroup $H$ of $G$ provides an approximate description of the group theoretic structure of $H$.

\begin{proposition}\label{prop:exceptional}
There is no nontrivial $2$-design with prime replication number $r$ admitting a flag-transitive almost simple automorphism group whose socle is an exceptional finite simple group of Lie type.
\end{proposition}
\begin{proof}
Suppose that $\Dmc$ is a nontrivial $2$-design with prime replication number $r$. Let $G$ be a flag-transitive almost simple automorphism group of $\Dmc$ with simple socle $X$, where $X$ is a finite exceptional simple group of Lie type. Then by Proposition~\ref{prop:flag} and Corollary~\ref{cor:large}, the point-stabliliser $H=G_{\alpha}$ is a large maximal subgroup of $G$. Suppose first that $H$ is a parabolic subgroup. If $X\neq E_{6}(q)$, then Lemma~\ref{lem:subdeg} implies that $r$ divides $\lambda p^{c}$ for some positive integer $c$, and so $r=p$, and hence by inspecting the value of $v$ in each case (see for example \cite[Table 5]{a:ABD-Exp}), we observe that $v>r^{2}$, which is a contradiction. If $X= E_{6}(q)$, then by Remark~\ref{rem:subdeg}, we only need  to consider the cases where the Levi factor of $H$ is of type $\Al_{1}\Al_{4}$ or $\D_{5}$ in which cases $r\leq q^{5}+1$ and $v>q^{14}+1$, and hence $v>r^{2}$, which is a contradiction. Therefore, $H$ is not parabolic, and hence we apply \cite[Theorem 1.6]{a:ABD-Exp} and obtain $H\cap X$ as listed in the second column of Table~\ref{tbl:large-exc}. By the same manner as in the classical groups, for each possible subgroup $H$, we can obtain a lower bound $l_{v}$ for $v$ and an upper bound $u_{r}$ for $r$, and then we easily observe that $\lambda v<r^{2}$ cannot hold. For example, if $X=F_4(q)$ and $H$ is of type $\Omega_9(q)$, then by Lemma \ref{lem:six}(c), we have that $q^{16}<q^{8}(q^{8}+q^{4}+1)=v<r^{2}\leq (q^{4}+1)^{2}$, and so $q^{16} <(q^{4}+1)^{2}$, which is impossible.
\end{proof}

\section*{Acknowledgements}

The authors would like to thank anonymous referees for providing us helpful and constructive comments and suggestions.
%\section*{References}

%\bibliographystyle{abbrv}
%\bibliography{references}

\end{document}